\documentclass{article}

\usepackage{geometry}                
\usepackage{graphicx}
\usepackage{wrapfig}
\usepackage{amssymb}
\usepackage{epstopdf}
\usepackage{enumerate}
\usepackage{gensymb}
\usepackage{amsmath}
\usepackage{amsthm}
\usepackage{ulem}
\usepackage[usenames, dvipsnames]{color}

\usepackage{algpseudocode}
\usepackage{algorithm}
\algdef{SE}[SUBALG]{Indent}{EndIndent}{}{\algorithmicend\ }%
\algtext*{Indent}
\algtext*{EndIndent}

\usepackage{amssymb}
\newtheorem{theorem}{Theorem}
\newtheorem{remark}{Remark}
\newtheorem{lemma}{Lemma}

\newtheorem{proposition}{Proposition}

\usepackage[utf8]{inputenc}
\usepackage[english]{babel}
\usepackage{leftidx}
\usepackage{nomencl}

\title{Generalized multiscale finite element method for the steady state linear Boltzmann equation}

\author{Eric Chung, \thanks{Department of Mathematics, The Chinese University of Hong Kong, Hong Kong SAR. Email: tschung@math.cuhk.edu.hk}
Yalchin Efendiev, \thanks{Department of Mathematics \& Institute for Scientific Computation (ISC), Texas A\&M University,
		College Station, Texas, USA. Email: efendiev@math.tamu.edu}
Yanbo Li \thanks{Department of Mathematics, Texas A\&M University,
		College Station, Texas, USA. Email: lyb@tamu.edu}
and Qin Li \thanks{Department of Mathematics, University of Wisconsin, Madison, USA. Email: qinli@math.wisc.edu}
}

\begin{document}
	\maketitle{}
	\begin{abstract}
The Boltzmann equation, as a model equation in statistical mechanics, is used to describe the statistical behavior of a large number of particles driven by the same physics laws. Depending on the media and the particles to be modeled, the equation has slightly different forms. In this article, we investigate a model Boltzmann equation with highly oscillatory media in the small Knudsen number regime, and study the numerical behavior of the Generalized Multi-scale Finite Element Method (GMsFEM) in the fluid regime when high oscillation in the media presents. The Generalized Multi-scale Finite Element Method (GMsFEM) is a general approach~\cite{chung2016adaptive} to numerically treat equations with multi-scale structures. The method is divided into the offline and online steps. In the offline step, basis functions are prepared from a snapshot space via a well-designed generalized eigenvalue problem (GEP), and these basis functions are then utilized to patch up for a solution through DG formulation in the online step to incorporate specific boundary and source information. We prove the wellposedness of the method on the Boltzmann equation, and show that the GEP formulation provides a set of optimal basis functions that achieve spectral convergence. Such convergence is independent of the oscillation in the media, or the smallness of the Knudsen number, making it one of the few methods that simultaneously achieve numerical homogenization and asymptotic preserving properties across all scales of oscillations and the Knudsen number.

		
	\end{abstract}
	\section{Introduction}
The Boltzmann equation is a fundamental model in statistical mechanics. It traces the evolution of the distribution function on the phase space, and describes the dynamics of a large number of particles that follow the same physics rules via a statistical manner. The equation encodes the particles' free transport and their interactions with the media and each other. Depending on the physics the particles follow, the interaction term may differ, but to a large extent, many particles, including neutrons, photons and phonons interact mainly with the media, making the collision term linear. The dynamics then can be described by the linear Boltzmann equation:
	
\begin{equation}
\partial_tu + v\cdot\nabla u = \sigma\mathcal{S}u(x,v) - \eta u\,,\quad (x,v)\in\Omega\times V\,.
\end{equation} 
	
In the equation, $u$ is a function on the phase space $(x,v)\in\Omega\times V$. The evolution is governed by $v\cdot\nabla u$, a free transport term, and the terms on the right side of the equation that represent the ``collision" and quantify the particles' interactions. These interactions include a pure absorption term $\eta u$ where $\eta$ is the absorption coefficient, and a scattering term $\sigma\mathcal{S}u$. The specific form of the operator $\mathcal{S}$ varies from particle to particle, and it is typically a functional independent of $x$. The strength of the interactions are governed by the size of $\sigma$ and $\eta$. For photons specifically, the radiative transfer equation is used, these coefficients are termed the optical thickness. In this paper, for simplicity, we take $\mathcal{R}$ to be:
\begin{equation}\label{eqn:op_L}
\mathcal{R}u(x,v) = \int_Vu(x,v')dv' - u(x,v)\,,
\end{equation}
where $dv$ is a normalized measure: $\int dv = 1$, and we set $\eta=1$.

The equation demonstrates different behavior in different regimes. One particularly interesting regime is called the diffusion regime, in which the scattering coefficient is extremely strong and the pure absorption term is weak. Mathematically, consider the steady state case, we model the equation to:
\begin{equation}\label{eqn:Boltzmann_kn}
v\cdot\nabla u + \epsilon \, u= \frac{1}{\epsilon}\sigma\mathcal{R}u(x,v)\,.
\end{equation}

In this equation, $\epsilon$ is termed the Knudsen number, and it characterizes the ratio of the mean free path and the typical domain length. Physically it reflects the number of collision a normal particle experiences inside $\Omega$ before emitting. When $\epsilon$ is small, the number of collision per particle is large, meaning the particle gets scattered many times before emitting, and thus some kind of averaging effects take place, and the local equilibrium is achieved. In the case of~\eqref{eqn:op_L}, the equilibrium reads:
\begin{equation}\label{eqn:equilibrium}
u(x,v)\sim \rho(x)\,,
\end{equation}
and through asymptotic analysis, one could mathematically derive that $\rho$ satisfies the diffusion equation:
\begin{equation}\label{eqn:diff_limit}
C\nabla\cdot(\frac{1}{\sigma}\nabla\rho) = \rho\,,
\end{equation}
where $C$ depends on the dimension.
	
The convergence from~\eqref{eqn:Boltzmann_kn} to the asymptotic limit~\eqref{eqn:diff_limit} was conjectured in~\cite{Bossoussan-Lions-Papanicolou} and was made rigorous in~\cite{Bardos-Sentos-Sentis} for periodic boundary condition. In~\cite{Wu-Guo} the authors studied the boundary layer effect with geometric corrections, and the asymptotic convergence rate was shown to degrade~\cite{Li-Lu-Sun, JSP}.

However, all the rigorous proofs are done assuming certain smoothness of $\sigma$. In particular, it is assumed that $\sigma$ is sufficiently smooth. At the current stage, very limited works have been done when oscillations present in the media. Denote $\delta$ the small scale in the media, we rewrite our equation as:
\begin{equation}\label{eqn:Boltzmann_kn_rewrite}
v\cdot\nabla u + \epsilon \, u= \frac{1}{\epsilon}\sigma^\delta\mathcal{R}u(x,v)\,,
\end{equation}
where $\sigma^\delta(x) = \sigma(x,\frac{x}{\sigma})$ to explicit reflects the fast variable $\frac{x}{\sigma}$ dependence. On the theoretical level, to our best knowledge, except a few cases~\cite{Goudon-Mellet}, the theory is largely in lack, except a few cases~\cite{Goudon-Mellet}, and to a large extend, we do not yet know the resonance of the two parameters, and how they contribute in the asymptotic limits of the equation. And on the computation level, the only numerical study aware to the authors is presented in~\cite{ll16} where the limits are taken in order: $\delta \ll \epsilon\ll 1$.

The problem is very challenging on the numerical level. The small $\epsilon$ makes the collision term extremely stiff, bringing ill-conditioning to the associated discrete system, and thus severe stability issue; and the small $\delta$ brings wild oscillations to the media and the solution, and for accuracy of the numerical solution, high resolution is needed and small discretization is necessary, driving up the numerical cost.

This is certainly unaffordable, especially in the zero limit of $\epsilon$ and $\delta$. The main goal of this paper is to develop a general numerical treatment that could deal with the equation with a wide range of $\epsilon$ and $\delta$, and perform uniformly well, with the error term independent on the small parameters.

The approach we take is in line of Generalized Multi-scale Finite Element Method (GMsFEM). This is an offline-online framework that builds a good set of local basis functions during the offline step and patches local solutions up in the online step, similar to the original Multi-scale Finite Element Method (MsFEM). One main feature of GMsFEM is its basis selection procedure in the offline step where a special generalized eigenvalue problem is designed. This special generalized eigenvalue problem encodes the oscillations and the ill-conditioning of the problem.

More explicitly, like many other multi-scale methods, we build nested grids with coarse grid $H$ and fine grid $h$ satisfying $H\gg\epsilon\gg h$. In the offline step, local basis functions are constructed within coarse mesh $H$ on fine mesh $h$ that capture fine scale structure and preserve the heterogeneities in the media; and in the online step, the basis functions are patched up through Galerkin framework~\cite{arbogast2004analysis, chu2010new, engquist2013heterogeneous, efendiev2011multiscale, efendiev2009multiscale, efendiev2004multiscale, ghommem2013mode, chung2014generalized, chung2013sub,GMsFEM-elastic,elastic-jcp,aarnes04,jennylt03}. Online step is rather standard and different methods give various algorithms in the offline step. What makes GMsFEM favorable is indeed its offline step, in which the full list of $a$-harmonic functions are collected, and then the most ``representative" modes are selected through a specially designed generalized eigenvalue problem (GEP). The definition of the matrices in the GEP is associated with the final error term, which permits certain spectral error decay. We should mention GMsFEM was initially used for elliptic equations containing strong heterogeneous media, a topic about which the literature is extremely rich. For this particular problem, there is another category of method: upscaling-type methods. In upscaling methods, either locally or globally an effective media is numerically prepared so that equations could be computed on coarser grids with the effective media replacing the heterogeneous one~\cite{durlofsky1991numerical, wu2002analysis,numerical-homo,2d-waves}. But this approach is not going to be pursued in this paper.

As a framework, the GMsFEM approach is rather easy to use, and the main mathematical challenge, when utilized for tackle different equations, is to develop the right GEP. For the linear Boltzmann equation with heterogeneous media, we frame the problem in the discontinuous Galerkin setting, and are able to find two matrices that resemble the mass and stiffness matrices in the GEP of the elliptic equations, which allows us to show the optimality of the basis functions with physically meaning definition of the norm for the error. As the standard approach, these basis functions  are then used in the online computation.

The paper will be organized in the following way. In Section 2, we introduce some preliminaries. Both discrete ordinates, the standard kinetic solver, and GMsFEM for elliptic equation will be presented. Some properties will be presented along. We present the algorithm in Section 3, which is further divided into two subsection introducing offline and online procedures. Section 4 contains the analysis where we present the wellposedness, and convergence results. The small $\epsilon$ limit of the method will also be discussed. Numerical results are to be shown in Section 5.
	
To end the introduction, we comment that the scaling problem studied in this paper is not mathematically artificial. In fact, as one redefines $x\to\frac{x}{\epsilon}$, $\sigma(x)$ should have been automatically changed to highly oscillatory media $\sigma(x/\epsilon)$. Another practical example is to inject light into crystals, where the radiative transfer equation (one particular linear Boltzmann equation) is utilized. In this case, the periodic crystal structure should be encoded in the media and the period that corresponds to $\delta$ in our math formulation, is expected to be small.
	
\section{Preliminaries}\label{sec:pre}
In this section we prepare some basic important concepts. In particular, we will first present the discrete ordinate method for the linear Boltzmann equation, and then give a brief account of the Generalized Multiscale Finite Element Method (GMsFEM). They are the building blocks for the algorithm designed in this paper.	
	
\subsection{$S_N$ Boltzmann Equation}\label{sec:S_N}
The Boltzmann equation gives a statistical description of particle dynamics. Its extensive use in all kinds of engineering problems brought its great popularity, and literature on both theory and numerics has been very rich. Among all numerical methods developed for the Boltzmann equation, the discrete ordinate method stands out for its simplicity and intuitiveness, and is the method we will use in our GMsFEM. Essentially it discretizes the velocity domain, and the semi-discrete system is a coupled PDE in the physical space.
	
We start the discussion with the following model equation:
\begin{equation}\label{Boltzmann_original}
\begin{split}
\mathbf{v}\cdotp \nabla u(x,\mathbf{v})+\epsilon u(x,\mathbf{v}) &\;=\;\frac{1}{\epsilon a^\delta}\mathcal{R}u(x,\mathbf{v}) \quad  \text{ in } \Omega\times S^1\\
u(x,\mathbf{v}) &\;=\;  g(x,\mathbf{v})  \quad\quad\quad\; \text{ on }\Gamma^-
\end{split}\,,
\end{equation}
where $x\in\Omega\subset\mathbb{R}^2$ is a bounded domain with a Lipschitz boundary $\partial\Omega$. The velocity is $v\in S^1$, the unit circle. The media $a^\delta$ presents fine scale structure at $\delta$ order, and the stiffness of the collision operator $\mathcal{R}$ is determined by $\frac{1}{\epsilon}\gg 1$. We have the inflow boundary condition, with the inflow data $g(x,\mathbf{v})$ defined on $\Gamma^-$, a collection of coordinates on the boundary with velocity pointing into the domain:
	\[
	\Gamma^-=\left\lbrace (x,v)\in \partial\Omega\times S^1\,|\, \mathbf{v}\cdot \mathbf{n}_x<0 \right\rbrace \,.
	\]
	Here $\mathbf{n}_x$ is the unit outer normal direction at $x\in\partial\Omega$. For simplicity, we use the model collision operator with homogeneous scattering coefficient:
	$$
	\mathcal{R}u(x,\mathbf{v})=\overline{u}(x)-u(x,\mathbf{v})=\frac{1}{2\pi}\int_{S^1}u(x,\mathbf{v})d\mathbf{v}-u(x,\mathbf{v})\,.
	$$
	
	The discrete ordinate method, denoted by $S_N$, is a standard method to discretize the velocity domain. One first sample $m$ quadrature points on $S^1$ and each sample point is associated with a weight, denoted by: $\{(\mathbf{v}_i,\alpha_i), i = 1,. . . ,m\}$, where $\mathbf{v}_i$ are the quadrature points and $\alpha_i$ are the corresponding positive weights. These quadrature points and weights are chosen so that:
\begin{equation}\label{eqn:weight_cond}
\sum_{i=1}^{m}\alpha_i=1,\quad \text{ and }\quad\frac{1}{2\pi}\int_{S^1}u(x,\mathbf{v})d\mathbf{v}\approx \sum_{i=1}^{m}\alpha_i u(x,\mathbf{v}_i)\,.
\end{equation}
The equation then will be discretized into semi-discrete system. Let $u_i(x)=u(x,\mathbf{v}_i)$, the integro-differential~\eqref{Boltzmann_original} is then transformed into a system of $m$ coupled partial differential equations:
\begin{equation}\label{main}
\begin{split}
\mathbf{v}_i\cdot \nabla u_i+\epsilon u_i+ \frac{1}{\epsilon a^\delta}\left( u_i-\sum_{j=1}^m \alpha_j u_j \right) & \;=\;  0\quad \quad\quad \text{in}\;\Omega\,,\\
u_i & \;=\;  g_i \quad \quad\;\;\;\text{on}\;\Gamma^-\,,
\end{split}
\end{equation}
where $g_i=g(x,\mathbf{v}_i)$ is the inflow boundary data. Denote
\begin{equation}\label{a_ij}
a_{ij}=\left\{\begin{matrix}
\alpha_i-\alpha_i^2, & i=j,\\
-\alpha_i\alpha_j, & i\neq j.
\end{matrix}\right.\,.
\end{equation}	
Then~\eqref{main} is further simplified to:
\begin{equation}\label{main_a}
\mathbf{v}_i\cdot \nabla u_i+\epsilon u_i+\frac{1}{\epsilon a^\delta\alpha_i}\sum_{j}a_{ij} u_j  \;=\;  0\,.
\end{equation}

Since $\{a_{ij}\}$ is basically the discrete version of the collision operator $-\mathcal{R}$, it resembles the properties of $\mathcal{R}$. In particular, the matrix is positive semi-definite with a known kernel.
\begin{proposition}\label{prop:semi-definite}
Define a matrix $\mathsf{A}$ so that $\mathsf{A}_{ij} = a_{ij}$, we claim:
\begin{itemize}
\item $\mathsf{A}$ is positive semi-definite.
\item $\mathsf{u}^\top\mathsf{A}\mathsf{u}=0$ if and only if $\mathsf{u}=(u_1,u_2,...,u_m)$ is isotropic, i.e., $u_1=u_2=\cdots=u_m$.
\item $\mathsf{v}^\top\mathsf{A}\mathsf{u}=0$ if either $\mathsf{u}$ or $\mathsf{v}$ is isotropic.
\end{itemize}
\end{proposition}

\begin{proof}
The computation is straightforward:
\begin{align*}
\mathsf{u}^\top\mathsf{A}\mathsf{u}=\sum_{i,j=1}^m a_{ij}u_ju_i=&\sum_{i=1}^m(\alpha_i-\alpha_i^2) u_i^2-2\sum_{i< j}\alpha_i\alpha_ju_i u_j\\
		=&\sum_{i=1}^m\alpha_i\sum_{j\not= i}\alpha_j u_i^2-2\sum_{i< j}\alpha_i\alpha_ju_i u_j\\
		=&\sum_{i< j}\alpha_i\alpha_j\left(u_i-u_j \right)^2\geq 0\,. 
		\end{align*}
The equal sign is achieved only when $u_i = u_j$ for all $i\neq j$.

To show the third bullet point, we note that the matrix $\mathsf{A}$ is symmetric, it suffices to assume that $\mathsf{u}$ is isotropic: $u_1=u_2=\cdots=u_m=\overline{u}$. Then:
\begin{align*}
\mathsf{v}^\top\mathsf{A}\mathsf{u}=\sum_{i,j=1}^m a_{ij}u_jv_i=&\sum_{i,j=1}^m a_{ij}\overline{u}v_i\sum_{i=1}^m\overline{u}v_i\sum_{j=1}^ma_{ij}\\
=&\sum_{i=1}^m\overline{u}v_i(\alpha_i-\alpha_i^2+\sum_{j\neq i}-\alpha_i\alpha_j)\\
=&\sum_{i=1}^m\overline{u}v_i\alpha_i(1-\alpha_i-\sum_{j\neq i}\alpha_j)=0\,.
\end{align*}
where we have used the weight condition~\eqref{eqn:weight_cond}: $\sum_i\alpha_i=1$.
\end{proof}

\subsection{Generalized Multi-scale Finite Element Method}\label{sec:GMsFEM}
The discrete ordinate method is used to discretize the velocity domain, and for the spatial domain, we follow the GMsFEM approach, which, by choosing ``optimal basis functions" via a special design of a GEP, we can obtain a reduced model that is robust for all values of $\epsilon$ and $\delta$. For the completeness of the paper, we now present a general idea of GMsFEM, and its application to the heterogeneous Boltzmann equation will be discussed in details in Section 3.

The GMsFEM uses two stages: offline and online. In the offline stage, a small dimensional approximation space is constructed to solve the global problem for any external source on a coarse grid, which does not need to resolve any scales of the media and solution. The offline stage consists of two main concepts. The snapshot space,  $V_{\text{snap}}^{i}$, is constructed for a generic coarse element $K_i$. The snapshot solutions are used to compute local multiscale basis functions. An appropriate snapshot space can
\begin{itemize}
\item provide a faster convergence,
\item provide problem relevant restrictions on the coarse spaces (e.g., divergence free solutions),
\item reduce the cost associated with constructing the offline spaces.
\end{itemize}
Standard choices of snapshot spaces (see \cite{chung2016adaptive}) are (1) all fine-grid functions; (2) snapshots of local solutions; (3) oversampling snapshots of local solutions; and (4) force-based snapshots. In this paper, we will use snapshots of local solutions. 

More specifically, these are functions $\eta_l^{(i)}$ that satisfy
\[
\mathcal{L}( \eta_{l}^{(i)})=0\ \ \text{in} \ K_i
\]
subject to some boundary conditions, where $\mathcal{L}$ is the differential operator under consideration, and $l$ is the index for the boundary condition. One can use all fine grid delta functions as boundary conditions or randomized boundary conditions~\cite{chung2016adaptive, randomized2014}.

	The offline space, $V_H$, is computed for each
	$K_i$
	(with elements of the space denoted $\psi_l^{(i)}$).
	We perform a spectral decomposition in the snapshot space and select
	the dominant eigenfunctions
	(corresponding to the smallest eigenvalues)
	to construct the offline (multiscale) space.
	The convergence rate of the resulting method is proportional
	to $1/\Lambda_*$, where
	$\Lambda_*$ is the smallest eigenvalue that the corresponding eigenvector
	is not included in the multiscale space.
	We would like to select local spectral problem such that
	we can remove many small eigenvalues with fewer multiscale basis
	functions. 
	The choice of spectral problems is usually problem dependent and is based on convergence analysis. 
	In general, 
	the error is decomposed into coarse subdomains.
	The energy functional corresponding to
	the domain $\Omega$ is denoted by $a_\Omega(u,u)$, e.g.,
	$a_\Omega(u,u) = \int_\Omega \kappa \nabla u\cdot \nabla u$.
	Then,
	\begin{equation}
	\begin{split}
	a_\Omega(u-u_H,u-u_H)\preceq
	\sum_K a_K(u^K-u_H^K,u^K-u_H^K),
	\end{split}
	\end{equation}
	where $K$ are coarse regions ($K_i$),
	$u^K$ is the localization of the solution.
	The local spectral problem is chosen to bound
	$ a_K(u^K-u_H^K,u^K-u_H^K)$.
	We seek
	the subspace $V_{H}^{i}$
	such that for any
	$\eta\in V_{\text{snap}}^{i}$,
	there exists
	$\eta_0\in V_{H}^{i}$ with,
	\begin{equation}
	\label{eq:off1}
	a_{K_i}(\eta-\eta_0,\eta-\eta_0)\preceq {\beta}
	s_{K_i}(\eta-\eta_0,\eta-\eta_0),
	\end{equation}
	where
	$s_{K_i}(\cdot,\cdot)$ is an auxiliary bilinear form, and $\beta$ is an accuracy parameter.
	The auxiliary bilinear form needs to be chosen such that the solution
	is bounded in the corresponding norm.
	
	Finally, in the online stage, the space $V_H$ is used together with a suitable coarse grid discretization
	to solve the problem. The same space $V_H$ is used for all input sources. 

\section{GMsFEM for heterogeneous Boltzmann equation}
We now apply the GMsFEM approach to numerically study the heterogeneous Boltzmann equation, expressed in the discrete ordinate system~\eqref{main}.

The numerical difficulties in solving this equation are summarized as follows. First, the media $a^\delta$ is highly oscillatory. The fine structure oscillates at the scale of $\delta$ injects high heterogeneities to $u_i$. In order to capture these details, the mesh size $h$ has to be smaller than $\delta$, which in turn brings prohibitive numerical cost. Secondly, the operator $\mathcal{L}$ is scaled by $\frac{1}{\epsilon}$, and in the zero limit of $\epsilon$, the term is extremely stiff, and this brings concern in stability. It is our aim in this paper to develop a multiscale method that can address these issues. In particular, inspired by GMsFEM, we will design a numerical method that relies on offline basis construction and online basis patching procedure, and its numerical error has limited dependence on the two small parameters.

We will construct nested grids and call $\mathcal{T}^h$ the partition of $\Omega$ into fine finite elements, and $\mathcal{T}^H$ the partition into coarse elements, where $h$ and $H$ are the fine and coarse mesh sizes respectively. For simpler notation, we consider rectangular coarse elements as shown in Figure \ref{partition}. The basis functions and discretization are based on the coarse grid, and the fine grid is used to numerically compute the basis functions. We also denote the collection of coarse edges $\mathcal{E}_H$, and $\mathcal{E}_H^0=\mathcal{E}_H\backslash \partial \Omega$ the collection of coarse edges in the interior of the domain.

\begin{figure}[ht]
\begin{minipage}[t]{0.5\textwidth}
\centering
\includegraphics[width=2.6in]{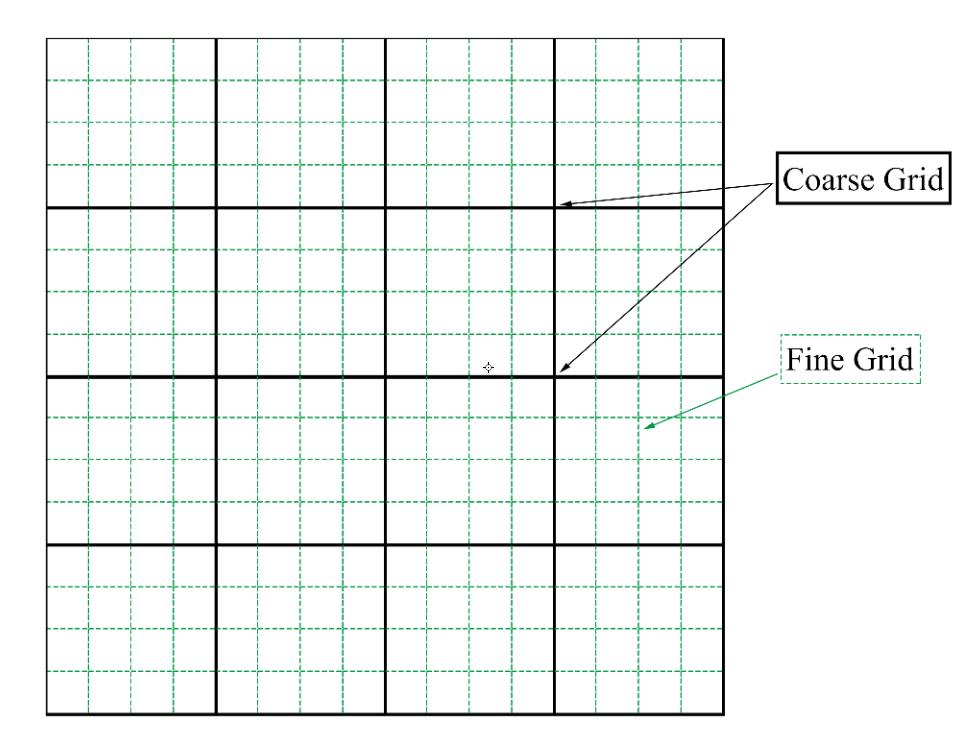}
\end{minipage}%
\begin{minipage}[t]{0.5\textwidth}
\centering
\includegraphics[width=2.95in]{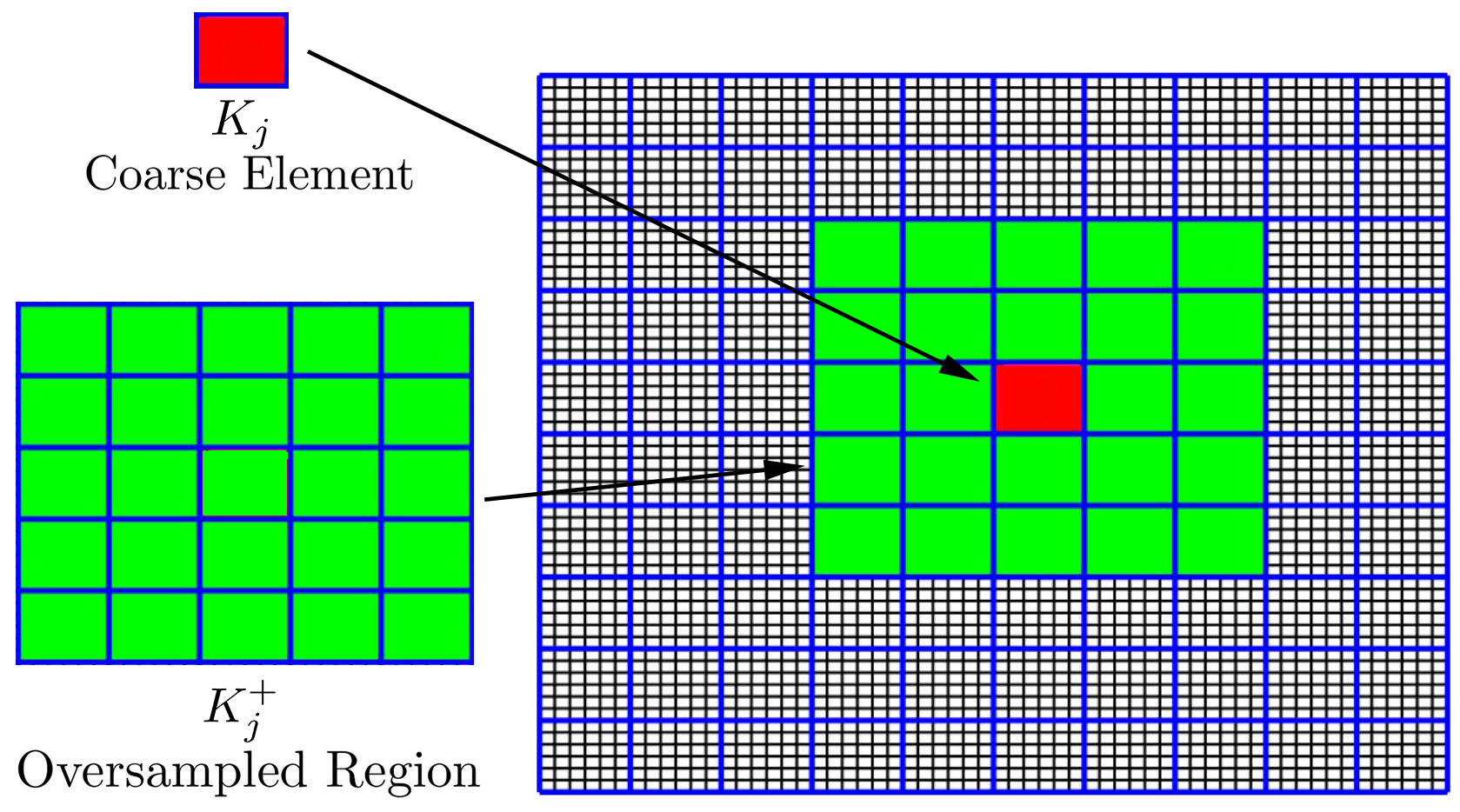}
\end{minipage}
\caption{Left: an illustration of fine and coarse grids. Right: an illustration of a coarse neighborhood and a coarse element.}
\label{partition}
\end{figure}

The discontinuous Galerkin method allows one to pick different values of the solution on different sides of the edges. Suppose two adjacent coarse blocks $\tau_i$ and $\tau_j$ share an edge, and that $\tau_i$ is the upwind block, then we denote $w^+=w|_{\tau_i}$ and $w^-=w|_{\tau_j}$. Notice that depending on the direction of a specific $\mathbf{v}_i$, different block could be picked as the upwind block, as shown in Figure \ref{upwind}. 
\begin{figure}[ht]
\centering
\begin{minipage}[t]{1\textwidth}
\centering
\includegraphics[width=2.5in]{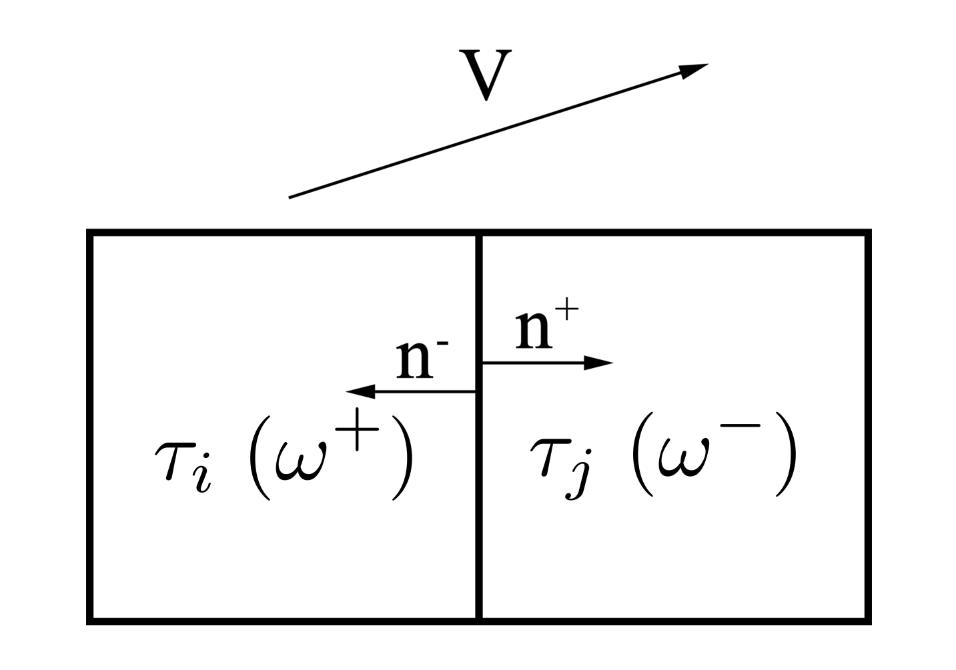}
\end{minipage}%
\caption{An illustration of upwind and downwind blocks.}
\label{upwind}
\end{figure}

For the fine scale approximation, we choose the discrete function space to be:
\begin{align*}
V_h= & \left\{ v\in L^2(\Omega)\,|\; v|_\tau\in Q_1(\tau),\;\forall \tau\in \mathcal{T}^h\text{ and } v|_K\in C^0(K),\;\forall K\in \mathcal{T}^H  \right\}\,,
\end{align*}
and we seek for numerical solution such that
\[
u_h=(u_{h,1},u_{h,2},...,u_{h,m})\in (V_h)^m\,.
\]
This means the numerical solution for each $u_{h,i}$, when confined in each fine grid, is a linear function, and is continuous function across coarse grids. In the variational formulation: for all $i=1,2,...,m$, we have:
	\begin{equation}
	\label{fine}
	\begin{split}
	& - \int_\Omega u_{h,i}\nabla w_i\cdot\mathbf{v}_i  + \sum_{e\in\mathcal{E}_H^0}\int_e  u_{h,i}^+[w_i]\cdot\mathbf{v}_i+\sum_{e\in{\Gamma^+}}\int_e  u_{h,i} w_i\mathbf{v}_i\cdot\mathbf{n}+\int_\Omega\epsilon u_{h,i}w_i\\
	+&\int_\Omega\frac{1}{\epsilon a^\delta}\left( u_{h,i}-\sum_{j=1}^m \alpha_ju_{h,j}\right) w_i=-\sum_{e\in{\Gamma^-}}\int_e  g_i w_i\mathbf{v}_i\cdot\mathbf{n}, \quad\forall w_i\in V_h\,,
	\end{split}
	\end{equation}
or with the definition of $\{a_{ij}\}$ in~\eqref{a_ij}, they could be summed up to
\begin{align}\label{eq:fine}
\begin{split}
&\sum_{i=1}^m\alpha_i\left( -\int_\Omega  u_{h,i}\nabla w_i\cdot\mathbf{v}_i + \sum_{e\in\mathcal{E}_H^0}\int_e  u_{h,i}^+[w_i]\cdot\mathbf{v}_i+\sum_{e\in{\Gamma^+}}\int_e  u_{h,i} w_i\mathbf{v}_i\cdot\mathbf{n}+\int_\Omega\epsilon u_{h,i}w_i\right) \\
&+\int_\Omega\frac{1}{\epsilon a^\delta}\sum_{i,j=1}^m a_{ij}u_{h,j}w_i=-\sum_{i=1}^m\alpha_i\sum_{e\in{\Gamma^-}}\int_e  g_i w_i\mathbf{v}_i\cdot\mathbf{n}, \quad\forall w\in (V_h)^m,
\end{split}
\end{align}
In the equation, we have used upwind approximation for $\mathbf{v}_i \cdot \nabla u_i$ and the jump operator $[\cdot]$ is defined by
\[
[w]=\left\{\begin{matrix}
w^-\mathbf{n}^-+w^+\mathbf{n}^+   & \text{on }\mathcal{E}_H^0\\ 
w^-\mathbf{n}^-   & \text{on }\Gamma^-\\ 
w^+\mathbf{n}^+   & \text{on }\Gamma^+
\end{matrix}\right.\,.
\]
	
For notational simplicity, we define two bilinear operators
	\begin{equation*}
	\begin{split}
	a(u,w)\;=\;&\sum_{i=1}^m\alpha_i a_i(u_i,w_i),
	\quad\text{with} \quad a_i(u_i,w_i)\;=\; - \int_\Omega u_i\nabla w_i\cdot\mathbf{v}_i + \sum_{e\in\mathcal{E}_H^0}\int_e u_i^+[w_i]\cdot\mathbf{v}_i+\sum_{e\in{\Gamma^+}}\int_e u_i w_i\mathbf{v}_i\cdot\mathbf{n}\,,\\
	l(u,w)  \;=\;&\sum_{i=1}^m\alpha_i\int_\Omega\frac{1}{\epsilon a^\delta}\left( u_i-\sum_{j=1}^m \alpha_j u_j \right) w_i+\sum_{i=1}^m\alpha_i\int_\Omega\epsilon u_iw_i=\int_\Omega\frac{1}{\epsilon a^\delta}\sum_{i,j=1}^m a_{ij} u_j w_i+\sum_{i=1}^m\alpha_i\int_\Omega\epsilon u_iw_i,
	\end{split}
	\end{equation*}
	and a linear operator:
	\begin{equation*}
	F(w)\;=\;\sum_{i=1}^{m}\alpha_i F_i(w_i),\quad\text{with}\quad F_i(w_i)\;=\;-\sum_{e\in{\Gamma^-}}\int_e g_iw_i\mathbf{v}_i\cdot\mathbf{n}\,.
	\end{equation*}
With these notations, equation~\eqref{eq:fine} now writes
\begin{equation}\label{eq:coupled_fine}
a(u_h,w)+l(u_h,w)=F(w),\quad \forall w\in (V_h)^m\,.
\end{equation}

With $h\ll\text{min}\{\epsilon\,,\delta\}$, it is a standard result that $u_h \approx u$, with an error term of size $\mathcal{O}(h^2/\text{min}(\epsilon,\delta))$. For significantly small $h$, the function $u_h$ is considered as a reference solution in accessing the performance of our method.

However, using small $h$ that resolves $\epsilon$ and $\delta$ leads to a very big system that is numerically very costly. We would like to develop an algorithm that seeks for solution only on the coarse grid $H$ and the corresponding solution $u_H\approx u_h\approx u$. To do that, an offline-online procedure developed in~\cite{randomized2014} for elliptic equation, termed GMsFEM (Generalized Multiscale Finite Element Method) will be pursued. In the offline step, an approximate space $V_H$ is constructed to replace $V_h$. This newly constructed space would have much less degrees of freedom but preserves $V_h$'s important factors. The final multiscale solution will be computed in the online step where the boundary condition $g(x,\mathbf{v})$ will be taken into account to determine the degrees of freedom in $V_H$.

We quickly review the online stage in Section~\ref{sec:online}, and the complicated offline step will be discussed in detail in Section~\ref{sec:basis}.

\subsection{Online computation}~\label{sec:online}
In online stage, we will use the multiscale basis functions together with a coarse grid discretization to solve the given problem. The coarse grid discretization we used is a discontinuous Galerkin method with upwind flux. Assume that a multiscale finite element space $V_H=\text{span}\{\phi_p\}$ is determined, and this space, in some sense, approximates $(V_h)^m$. Then similar to the formulation as in~\eqref{eq:fine}, the solution will be sought in
\[
u_H=(u_{H,1},u_{H,2},...,u_{H,m})\in V_H
\]
so that
\begin{align}\label{global}
	\begin{split}
	&\sum_{i=1}^m\alpha_i\left( -\int_\Omega  u_{H,i}\nabla w_i\cdot\mathbf{v}_i + \sum_{e\in\mathcal{E}_H^0}\int_e  u_{H,i}^+[w_i]\cdot\mathbf{v}_i+\sum_{e\in{\Gamma^+}}\int_e  u_{H,i} w_i\mathbf{v}_i\cdot\mathbf{n}+\int_\Omega\epsilon u_{H,i}w_i\right) \\
	&+\int_\Omega\frac{1}{\epsilon a^\delta}\sum_{i,j=1}^m a_{ij}u_{H,j}w_i=-\sum_{i=1}^m\alpha_i\sum_{e\in{\Gamma^-}}\int_e  g_i w_i\mathbf{v}_i\cdot\mathbf{n}, \quad\forall w\in V_H\,.
	\end{split}
\end{align}
Similar to~\eqref{eq:coupled_fine}, we use a compact notation:
\begin{equation}\label{eq:coupled_global}
	a(u_H,w)+l(u_H,w)=F(w),\quad \forall w\in V_H\,.
\end{equation}
	
To implement the scheme above, we define the following matrices
\begin{equation}\label{eqn:A_L_b_matrix}
A_{pq} = a(\phi_{p},\phi_{q})\,,\quad L_{pq} = l(\phi_{p},\phi_{q})\,,\quad\text{and}\quad b_{p} = F(\phi_{p})\,.
\end{equation}
Then the multiscale solution $u_H$ is formulated as
\begin{equation}\label{eqn:online}
u_H=\sum_p U_p\phi_p\,,
\end{equation}	
where the coefficient vector $U$ solves $(A+L)U=b$.

\subsection{Construction of $V_H$}\label{sec:basis}
The key to the success of our method is the construction of the space $V_H$ on the coarse mesh during the offline stage. We will give the details here. 

As discussed in Section~\ref{sec:GMsFEM}, the offline step is further decomposed into two substages: constructing the snapshot space, and selecting modes associated with small eigenvalues. These two stages will be presented in Section~\ref{sec:snap} and Section~\ref{sec:off} respectively.

In the snapshot space construction stage, in each coarse region, the Boltzmann solution will be solved multiple times together with all possible boundary conditions resolved by the fine grid. This give a high dimensional space. However, some modes in the snapshot space are more important than the others, and they dominate the numerical solution. To identify these basis functions, a specially designed local spectral problem (generalized eigenvalue problem GEP) is formulated and solved. The modes that correspond to the smallest eigenvalues are selected for form $V_H$. The number of modes to be selected depends on the error tolerance and the eigenvalues of the GEP. The design of the local spectral problem is to encode the convergence error that is to be discussed in Section~\ref{sec:analysis}.

\subsubsection{Snapshot Space} \label{sec:snap}
We present the construction of the snapshot spaces in this subsection. The procedure is the same in each coarse element, and we take the coarse element $K_j$ as an example. The snapshot space for this particular element is denoted by $V_\text{snap}^j$. We use the notation $J^i(D)$ to denote the set of all nodes of the fine mesh $\mathcal{T}^h$ lying in the upwind part of $\partial D$ associated with velocity $\mathbf{v}_i$. And we also use $J(D)=\bigoplus_{i=1}^m J^i(D)$ to denote the union.  Then the snapshot space is simply the linear span of solutions to the local Boltzmann equation with delta function as boundary condition, namely:
	\begin{equation}\label{def:V_snap_j}
	V^j_\text{snap} = \{\leftidx{_n}\psi_{l}^{j}: n=1,...,m,\;x_l\in J(K_j)\}\,,
	\end{equation}
	where $\leftidx{_n}\psi_{l}^{j} =(\leftidx{_n}\psi_{l,1}^{j},\leftidx{_n}\psi_{l,2}^{j},...,\leftidx{_n}\psi_{l,m}^{j})$ solves:
	\begin{align}\label{snapshot2}
	\left\{\begin{array}{rl}
	\mathbf{v}_i\cdot\nabla\leftidx{_n}\psi_{l,i}^{j}+\epsilon \leftidx{_n}\psi_{l,i}^{j}+\frac{1}{\epsilon a^\delta}\left( \leftidx{_n}\psi_{l,i}^{j}-\sum_{q=1}^m \alpha_q\leftidx{_n}\psi_{l,q}^{j} \right) =0 & \text{in}  \quad K_j \quad\text{ for all }i=1,2,...,m\,,\\ 
	\leftidx{_n}\psi_{l}^j = \delta_l\mathbf{e}_n& \text{on} \quad J(K_j)\,.
	\end{array}\right.
	\end{align}
	Here we use multi-index Kronecker delta function $\delta_l\mathbf{e}_n$, where $\mathbf{e}_n$ is the standard basis in $\mathbb{R}^m$ and $\delta_l$ is the standard Kronecker delta function:
\[
\delta_l(x_k)=\left\{\begin{matrix}
1, & k=l\\ 
0, & k\neq l
\end{matrix}\right.,\quad x_k\in J(K_j)\,.
\]
This strategy is summarized in Algorithm \textproc{DetLocal}.

The full snapshot space is given by
\begin{equation}\label{def:V_snap}
V_\text{snap}=\bigoplus_{j}V_\text{snap}^j\,.
\end{equation}

\begin{remark}
Numerically to prepare all snapshot basis functions is hard. It requires the computation of local Boltzmann equation with a large number of possible incoming delta functions. To reduce the cost of computation, we use the idea of oversampling \cite{efendiev2009multiscale}. To do so, the local computational domain is slightly enlarged to $K_j^+$ (see Figure \ref{partition}), and a collection of random boundary condition is imposed on $K_j^+$. The low rank structure of the solution space allows one to correctly capture the range, even with a limited sampling. In particular, we define the snapshot space
\begin{equation}\label{def:V_snap_j_oversampling}
V^j_\text{snap} = \{\leftidx{_n}\psi_{l}^{j,+}|_{K_j}: n=1,...,m,\;l=1,...,k_j\}\,,
\end{equation}
where $k_j$ is the number of snapshot functions we could customize, and $\leftidx{_n}\psi_{l}^{j,+} =(\leftidx{_n}\psi_{l,1}^{j,+},\leftidx{_n}\psi_{l,2}^{j,+},...,\leftidx{_n}\psi_{l,m}^{j,+})$ solves:
\begin{align}\label{snapshot2_oversampling}
\begin{array}{rl}
\mathbf{v}_i\cdot\nabla\leftidx{_n}\psi_{l,i}^{j,+}+\epsilon \leftidx{_n}\psi_{l,i}^{j,+}+\frac{1}{\epsilon a^\delta}\left( \leftidx{_n}\psi_{l,i}^{j,+}-\sum_{q=1}^m \alpha_q\leftidx{_n}\psi_{l,q}^{j,+} \right) =0 & \text{in}  \quad K_j^+ \quad\text{ for all }i=1,2,...,m\,,\\ 
\leftidx{_n}\psi^{j}_l = r_l\mathbf{e}_n & \text{on} \quad J(K^+_j)\,.
\end{array}
\end{align}
where $r_l$ are random i.i.d. Gaussian sampling on $J(K^+_j)$. The solutions $\leftidx{_n}\psi_l^{j,+}$ confined on $K_j$ are then used to form the snapshot spaces. We remark that the use of randomized boundary conditions on oversampling domains is able to reduce the offline computational cost as there is no need to impose delta function boundary conditions as in (\ref{snapshot2}).

This strategy is summarized in Algorithm \textproc{RanLocal}.
\end{remark}

Similar to~\eqref{eq:coupled_fine} and \eqref{eq:coupled_global}, we can solve the snapshot solution $u_\text{snap} \in V_{\text{snap}}$ by the following equation:
\begin{equation}\label{eq:coupled_snap}
a(u_\text{snap},w)+l(u_\text{snap},w)=F(w),\quad \forall w\in V_\text{snap}\,.
\end{equation}
We note that the snapshot solution can be considered as a reference solution. The error of the snapshot solution is related to the approximation property of the snapshot space in the fine scale space. 

\subsubsection{Offline Space}\label{sec:off}
Now, we will present the construction of the solution space $V_H$, with the property we mentioned in \eqref{eq:off1}. In the end $V_H$, when confined on each coarse element, say $K_j$, will be a subspace of $V^j_{\text{snap}}$, and the construction of $V_H$ is amount to finding the most appropriate basis functions in $V^j_\text{snap}$ to be included. The procedure is further divided into two sub-steps: the energy minimizing oversampling (EMO), and a design of a generalized eigenvalue problem, as used in \cite{chung2018emo}.

We first denote the local oversampled snapshot space $\bigoplus_{K_i\subset K_j^+}V_\text{snap}^i$ by $V_\text{snap}^{j,+}$. Notice that, for a given coarse element $K_j$ and its corresponding oversampling region $K_j^+$, the space $V_\text{snap}^{j,+}$ is the union of all snapshot spaces $V_\text{snap}^i$ with the condition that $K_i\subset K_j^+$.

	Then the energy minimizing snapshots are calculated. For any snapshot function $\psi\in V_\text{snap}^j$, its energy minimizing extension $\widetilde\psi$ has the smallest energy in some norm $a_{\text{Energy}}^j(\cdot,\cdot)$ and is sought in the local oversampled snapshot space $V_\text{snap}^{j,+}$ with the constraint  $\widetilde{\psi}|_{K_j}=\psi|_{K_j}$. In mathematical expression, for any $\psi\in V_\text{snap}^j$, we seek for $\widetilde\psi\in V_\text{snap}^{j,+}$ so that
\begin{align}\label{tilde}
\begin{split}
\widetilde{\psi}&=\text{argmin}_ {\widetilde{\phi}\in V_\text{snap}^{j,+}} a_{\text{Energy}}^j(\widetilde{\phi},\widetilde{\phi})\\
\text{s.t. }\quad \widetilde{\psi}&=\psi \quad \text{ in } K_j
\end{split}
\end{align}
in which
\begin{align}\label{eqn:energy_extension}
a_{\text{Energy}}^j(\widetilde{\phi},\widetilde{\phi})=\sum_{i=1}^m\alpha_i\left( \int_{K_j^+}|\nabla\widetilde{\phi}_i|^2+\frac{1}{H}\sum_{e\in \mathcal{E}_H^0(K_j^+)}\int_e  [\widetilde{\phi}_i]^2\right) +\int_{K_j^+}\frac{1}{\epsilon a^\delta}\sum_{i,l=1}^m a_{il}\widetilde{\phi}_l\widetilde{\phi}_i\,. 
\end{align}
We notice that this construction is well-defined and the strategy is summarized in Algorithm \textproc{EnergyMin}. As one can see, $\widetilde{\psi}$ is an extension of $\psi$ onto the oversampling domain that achieves the minimum energy, defined in~\eqref{eqn:energy_extension} and this extension is crucial, as will be seen in the later analysis.

\begin{remark}
	This is about a stable decomposition property. It is important that the local basis functions satisfy
	a stable decomposition property. More precisely, the sum of local energies is bounded by the global energy. 
\end{remark}
Next, we define the two bilinear operators $a_{K_j}(\cdot,\cdot)$ and $s_{K_j}(\cdot,\cdot)$, mentioned in \eqref{eq:off1}. For simplicity of notation, we use $a^j(\cdot,\cdot)$ and $s^j(\cdot,\cdot)$ instead. For the element $K_j$, define:
	\begin{align}\label{spectral}
	\begin{split}
	a^j(\phi,\eta)=&\sum_{i=1}^m\alpha_i\left( \int_{K_j^+}\nabla\widetilde{\phi}_i\cdot\nabla\widetilde{\eta}_i+\frac{1}{H}\sum_{e\in \mathcal{E}_H^0(K_j^+)}\int_e[\widetilde{\phi}_i][\widetilde{\eta}_i]\right) +\int_{K_j^+}\frac{1}{\epsilon a^\delta}\sum_{i,l=1}^m a_{il}\widetilde{\phi}_l\widetilde{\eta}_i\,,\\
	s^j(\phi,\eta)=&\sum_{i=1}^m\alpha_i\left( \frac{1}{2}\sum_{K\subset K_j^+}\int_{\partial K}\left | \mathbf{v}_i\cdot\mathbf{n} \right |\widetilde{\phi}_i\widetilde{\eta}_i+\int_{K_j^+}\epsilon \widetilde{\phi}_i\widetilde{\eta}_i\right) +\int_{K_j^+}\frac{1}{\epsilon a^\delta}\sum_{i,l=1}^m a_{il}\widetilde{\phi}_l\widetilde{\eta}_i\,.
	\end{split}
	\end{align}
	Using the above bilinear forms, a spectral problem is defined. On $K_j$, we look for $\left( \phi_{k}^j,\lambda_{k}^j\right) \in V_\text{snap}^j\times \mathbb{R}$ such that
	\[
	a^j(\phi_{k}^j,\eta)=\lambda s^j(\phi_{k}^j,\eta), \;\;\;\forall \eta\in V_\text{snap}^{j}\,
	\]
	where the eigenvalues are ordered in the ascending way:
	\[
	\lambda_{j,1}\leq\lambda_{j,2}\leq\cdots\,.
	\]
	For implementation, we define the following matrices
	\begin{equation}\label{eqn:A_S_matrix}
	A_{pq}^j = a^j(\psi_{p}^j,\psi_{q}^j)\,,\quad\text{and}\quad S_{pq}^j = s^j(\psi_{p}^j,\psi_{q}^j)\,.
	\end{equation}
	Then the pair $\left( \phi_{k}^j,\lambda_{k}^j\right)$ is computed by solving
	\begin{equation}\label{eqn:GEP}
	A^j\mathbf{c}_k = \lambda_k S^j\mathbf{c}_k\,,\quad \text{with}\quad\phi_{k}^j = \sum_{p}{c}_{k,p}{\psi}_{p}^j\,.
	\end{equation}
	Suppose $L_j$ modes are used for each $K_j$. This strategy is summarized in Algorithm \textproc{LocalGEP}. The offline space $V_H$ is given by
	\begin{equation}\label{def:V_H}
	V_H^j = \text{span}\{\phi_{k}^j:k=1\cdots L_j\}\,,\quad \text{and}\quad V_H=\bigoplus_{j}V_H^j\,.
	\end{equation}
	This will be the approximation space for solving the system~\eqref{main} in the scheme~\eqref{eq:coupled_global}. 

	\subsection{Algorithm summary}
We finally summarize the algorithm. Largely speaking, we prepare the basis functions in the offline step, and patch them up in the online step. The offline step is further divided into preparing a snapshot space in which either all Green's functions are accumulated or a good random selection is obtained, and the basis selection step, in which the local generalized eigenvalue problem is computed and eigenfunctions with highest energies are chosen. These basis functions are ultimately used in the online step via the weak formulation~\eqref{eqn:online}. We summarize the procedure in Algorithm~\ref{alg:summary}.
\begin{algorithm}
  \caption{Multiscale solver for $\mathcal{L}u=0$ over $\Omega$ with $u = g$ on $\Gamma^-$}\label{alg:summary}
  \begin{algorithmic}[1]
    \State \textbf{Domain Decomposition}
    \Indent
    \State Partition domain into non-overlapping patches $\Omega= \bigcup_j K_j$. 
    \EndIndent
    \State \textbf{Offline Stage:}
    \Indent
    \State \textbf{Snapshot Space}
    \Indent
    \State Form snapshot space by calling $V^j_\text{snap}$=\textproc{DetLocal}($K_j$) or $V^j_\text{snap}$=\textproc{RanLocal}($K_j^+$).
    \EndIndent
    \State \textbf{Offline Space}
    \Indent
    \State Form offline space by calling $V^j_H$=\textproc{LocalGEP}($K_j$).
    \State $V_H=\bigoplus_{j}V_H^j=\text{span}\{\phi_p\}$.
    \EndIndent
    \EndIndent
    \State \textbf{Online Stage:}
    \Indent
    \State Use global inflow boundary data $g$ to determine $U$ using~\eqref{eqn:online}.
    \EndIndent
    \State \textbf{Return:} approximated global solution $u_H=\sum_p U_p\phi_p$.
  \end{algorithmic}
  \hrulefill
  \begin{algorithmic}[1]
    \Function{DetLocal}{$K_j$}
    \State	Prepare full list of multi-index Kronecker delta function $\delta_l\mathbf{e}_n$ on $J(K_j)$.
    \State Find $\leftidx{_n}\psi_{l}^{j}$ using~\eqref{snapshot2}.
    \State  Formulate $V^j_\text{snap} = \{\leftidx{_n}\psi_{l}^{j}: n=1,...,m,\;x_l\in J(K_j)\}$ according to~\eqref{def:V_snap_j}.
    \State  \textbf{Return:} Local snapshot space $V^j_\text{snap}$.
    \EndFunction
  \end{algorithmic}
  \hrulefill
  \begin{algorithmic}[1]
    \Function{RanLocal}{$K_j$}
    \State	Prepare $k_j$ random i.i.d.~Gaussian vector $r_l\mathbf{e}_n$ on $J(K_j^+)$.
    \State Find $\leftidx{_n}\psi_{l}^{j,+}$ using~\eqref{snapshot2_oversampling}.
    \State  Formulate $V^j_\text{snap} = \{\leftidx{_n}\psi_{l}^{j,+}|_{K_j}: n=1,...,m,\;l=1,...,,k_j\}$ according to~\eqref{def:V_snap_j}.
    \State  \textbf{Return:} Local snapshot space $V^j_\text{snap}$.
    \EndFunction
  \end{algorithmic}
  \hrulefill
  \begin{algorithmic}[1]
	\Function{LocalGEP}{$K_j$}
	\State $V_\text{snap}^{j,+}=\bigoplus_{K_i\subset K_j^+}V_\text{snap}^i$.
	\State Compute $\widetilde{\psi}\in V_\text{snap}^{j,+}$ using~\eqref{tilde}.
	\State Solve the generalized eigenvalue problem~\eqref{eqn:GEP} for $V_H^j = \text{span}\{\phi_{k}^j:k=1\cdots L_j\}$.
	\State \textbf{Return:} Offline space $V_H^j$.  \EndFunction
  \end{algorithmic}
\end{algorithm}

	\section{Analysis of the GMsFEM}\label{sec:analysis}
In this section, we will present some analysis of our GMsFEM.
In Section \ref{sec:well}, we will prove the well-posedness of the discrete system resulting from the GMsFEM, and 
in Section \ref{sec:conv}, we will prove the convergence of the method. 
Finally, in Section \ref{sec:small}, we will analyze the behavior of the method when $\epsilon$ is small.

\subsection{Well-posedness}\label{sec:well}
We first show the well-posedness of the GMsFEM~\eqref{eq:coupled_global}.
\begin{theorem}\label{wellposedness}
	Problem~\eqref{eq:coupled_global} has a unique solution, and the solution $u_H$ satisfies the following stability condition
	\begin{align}\label{eqn:wellposed_stability}
	\sum_{i=1}^{m}\alpha_i\left( \frac{1}{4}\sum_{e\in \mathcal{E}_H}\int_e\left | \mathbf{v}_i\cdot\mathbf{n} \right |[u_{H,i}]^2+\int_{\Omega}\epsilon u_{H,i}^2\right) +\int_\Omega\frac{1}{\epsilon a^\delta}\sum_{i,j=1}^m a_{ij}u_{H,j}u_{H,i} \leq\sum_{i=1}^{m}\alpha_i\sum_{e\in \Gamma^-}\int_e\left | \mathbf{v}_i\cdot\mathbf{n} \right |g_i^2\,.
	\end{align}
\end{theorem}
\begin{proof}
	Since the system \eqref{eq:coupled_global} is a square linear system, showing the existence and uniqueness is amount to proving that $a(\widehat{u},w)+l(\widehat{u},w)=0$ for all $w\in V_H$ only for trivial solution $\widehat{u}=0$.
	
	We will first prove the following inequalities
	\begin{equation*}
	l(u,u)\geq 0\,,\quad\text{and}\quad a(u,u) = \sum_ia(u_i,u_i)\geq 0\,,\quad\forall u\in V_H\,.
	\end{equation*}
	First, $l$ is non-negative since the matrix $(a_{ij})$ is a positive semi-definite matrix, as discussed in Proposition~\ref{prop:semi-definite}. Next, the non-negativity of $a(\cdot,\cdot)$ is shown below:
	\begin{align}\label{a_and_V}
	\begin{split}
	a(u_i,u_i)=&-\int_\Omega u_i\nabla u_i\cdot\mathbf{v}_i +
	\sum_{e\in\mathcal{E}_H^0}\int_e u_i^+[u_i]\cdot\mathbf{v}_i +\sum_{e\in{\Gamma^+}}\int_e u_i^2\mathbf{v}_i\cdot\mathbf{n}\\
	=&- \frac{1}{2}\sum_{\tau\in \mathcal{T}^h}\int_{\partial \tau}u_i^2\mathbf{v}_i\cdot\mathbf{n}+\sum_{e\in\mathcal{E}_H^0}\int_e u_i^+[u_i]\cdot\mathbf{v}_i+\sum_{e\in{\Gamma^+}}\int_e u_i^2\mathbf{v}_i\cdot\mathbf{n}\\
	=&-\frac{1}{2}\sum_{e\in{\Gamma^+}}\int_e\left| \mathbf{v}_i\cdot\mathbf{n}\right|u_i^2+\frac{1}{2}\sum_{e\in{\Gamma^-}}\int_e\left| \mathbf{v}_i\cdot\mathbf{n}\right|u_i^2+\frac{1}{2}\sum_{e\in\mathcal{E}_H^0}\int_e\left| \mathbf{v}_i\cdot\mathbf{n}\right|\left({u_i^-} ^2-{u_i^+} ^2\right) \\
	&+\sum_{e\in\mathcal{E}_H^0}\int_e \left| \mathbf{v}_i\cdot\mathbf{n}\right|u_i^+\left( u_i^+-u_i^-\right) +\sum_{e\in{\Gamma^+}}\int_e \left| \mathbf{v}_i\cdot\mathbf{n}\right|u_i^2\\
	=&\frac{1}{2}\sum_{e\in{\Gamma^+}}\int_e\left| \mathbf{v}_i\cdot\mathbf{n}\right|u_i^2+\frac{1}{2}\sum_{e\in{\Gamma^-}}\int_e\left| \mathbf{v}_i\cdot\mathbf{n}\right|u_i^2+\frac{1}{2}\sum_{e\in\mathcal{E}_H^0}\int_e\left| \mathbf{v}_i\cdot\mathbf{n}\right|\left({u_i^+}-{u_i^-}\right)^2\\
	=&\frac{1}{2}\sum_{e\in \mathcal{E}_H}\int_e\left | \mathbf{v}_i\cdot\mathbf{n} \right |[u_i]^2\geq 0.
	\end{split}
	\end{align}
	
	Assuming $a(\widehat{u},w)+l(\widehat{u},w)=0$ for any $w\in (V_h)^m$, then setting $w=\widehat{u}$, we have
	\begin{equation}\label{th1_1}
	\sum_{i=1}^{m}\alpha_i\left( \frac{1}{2}\sum_{e\in \mathcal{E}_H}\int_e\left | \mathbf{v}_i\cdot\mathbf{n} \right |[\widehat{u_i}]^2 +\int_{\Omega}\epsilon \widehat{u_i}^2\right) +\int_\Omega\frac{1}{\epsilon a^\delta}\sum_{i,j=1}^m a_{ij} \widehat{u_j} \widehat{u_i}= 0.
	\end{equation}
	According to Proposition~\ref{prop:semi-definite}, we have
	\begin{equation}
	\widehat{u_1}=\widehat{u_2}=\cdots=\widehat{u_m}=0\,,
	\end{equation}
	meaning $\widehat{u}=0$, and the solution to~\eqref{eq:coupled_global} is thus unique. For stability, we start with
	\begin{align*}
	F(u_H)=&-\sum_{i=1}^{m}\alpha_i\sum_{e\in{\Gamma^-}}\int_e g_i u_{H,i}\mathbf{v}_i\cdot\mathbf{n}\\
	\leq&\sum_{i=1}^{m}\alpha_i\sum_{e\in \Gamma^-}\int_e\left | \mathbf{v}_i\cdot\mathbf{n} \right |g_i^2+\frac{1}{4}\sum_{i=1}^{m}\alpha_i\sum_{e\in \mathcal{E}_H}\int_e\left | \mathbf{v}_i\cdot\mathbf{n} \right |[u_{H,i}]^2\,.
	\end{align*}
	Considering $a(u_H,u_H)+l(u_H,u_H) = F(u_H)$, we conclude with the stability inequality~\eqref{eqn:wellposed_stability}.
\end{proof}

We notice that the snapshot equation~\eqref{eq:coupled_snap} has the same structure, and the wellposedness is proved in the same way.

\subsection{Convergence Analysis}\label{sec:conv}
We now analyze the convergence of the proposed method. The goal of this section is to estimate the difference between the snapshot solution, $u_\text{snap}$, computed in~\eqref{eq:coupled_snap}, and the multiscale coarse solution, $u_H$, computed in~\eqref{eq:coupled_global}. To do so, we first define the following norms. We define the $V$-norm as:
\begin{equation}\label{def:V_norm}
\left \| u \right \|_V^2=\sum_{i=1}^m\alpha_i \left \| u_i \right \|_{V^i}^2\,,
\quad\text{with}\quad	\left \| u_i \right \|_{V^i}^2=\frac{1}{2} \sum_{e\in \mathcal{E}_H}\int_e\left | \mathbf{v}_i\cdot\mathbf{n} \right |[u_i]^2\,,
\end{equation}
and $W$-norm as:
\begin{equation}\label{def:W_norm}
\left \| u \right \|_W^2=\sum_{i=1}^m\alpha_i \left \| u_i \right \|_{W^i}^2\,,\quad\text{with}\quad	\left\| u_i\right\| _{W^i}^2=\frac{1}{2}\sum_{K_j}\int_{\partial K_j}\left| \mathbf{v}_i\cdot\mathbf{n}\right| u_i^2\,.
\end{equation}
We also extend them by incorporating the collision term:
\begin{equation}\label{def:tilde_norm}
\left\|u \right\|^2_{\widetilde{V}}=\left\|u \right\|^2_{V}+l(u,u)\,,\quad\text{and}\quad	\left\|u \right\|^2_{\widetilde{W}}=\left\|u \right\|^2_{W}+l(u,u)\,.
\end{equation}
The total energy is now defined by:
\begin{equation*}
\left \| u \right \|_\text{Energy}^2=\sum_{i=1}^m\alpha_i\left( \int_{\Omega}|\nabla u_i|^2+ \frac{1}{H}\sum_{e\in \mathcal{E}^0_H}\int_e[u_i]^2\right)+\int_\Omega\frac{1}{\epsilon a^\delta}\sum_{i,j=1}^m a_{ij} {u_j} {u_i}\,.
\end{equation*}
Note that we have following propositions

\begin{proposition}\label{a_and_V_proposition}
	$a(u,u) = \|u\|^2_{{V}},$ and ~$a(u,u)+l(u,u) = \|u\|^2_{\widetilde{V}}.$
\end{proposition}
\begin{proof}
	This proposition simply comes from the calculations in \eqref{a_and_V}.
\end{proof}

\begin{proposition}\label{norm_connection}
	If $u \in V_\text{snap}$, we have
	\begin{align}
	1.\quad&\|u\|^2_{\widetilde{W}}\leq \sum_{j} s^j(u|_{K_j},u|_{K_j}),\label{sj_and_W}\\
	2.\quad&\sum_{j} a^j(u|_{K_j},u|_{K_j})\leq M\|u\|^2_\text{Energy}.\label{aj_and_Energy}
	\end{align}
	Here $a^j$ and $s^j$ are bilinear operator defined in~\eqref{spectral}. $M=\max_{K,E} \{M_K,M_E\}$ where $M_K$ is the number of oversampled regions $K_j^+$'s which have nonempty intersection with coarse block $K$, and $M_E$ is the number of oversampled regions $K_j^+$'s whose interior coarse edges $\mathcal{E}_H^0(K_j^+)$ contains coarse edge $E$. They are both small numbers.
\end{proposition}
\begin{proof}
	We denote $u|_{K_j}$ by $u^j$. So $u^j\in V_\text{snap}^j$. According to~\eqref{tilde}, $u^j$ has an energy minimizing extension $\widetilde{u}^j\in V_\text{snap}^{j,+}$ that satisfies $\widetilde{u}^j=u^j$ in $K_j$. Then we have
	\begin{align*}
	&\sum_{i=1}^m\alpha_i\left( \frac{1}{2}\sum_{K_j}\int_{\partial K_j}\left | \mathbf{v}_i\cdot\mathbf{n} \right |(u_i^j)^2+\int_{K_j}\epsilon (u_i^j)^2\right) +\int_{K_j}\frac{1}{\epsilon a^\delta}\sum_{i,l=1}^m a_{il}{u_l^j}{u_i^j}\\
	=&\sum_{i=1}^m\alpha_i\left( \frac{1}{2}\sum_{K_j}\int_{\partial K_j}\left | \mathbf{v}_i\cdot\mathbf{n} \right |(\widetilde{u}_i^j)^2+\int_{K_j}\epsilon (\widetilde{u}_i^j)^2\right) +\int_{K_j}\frac{1}{\epsilon a^\delta}\sum_{i,l=1}^m a_{il}{\widetilde{u}_l^j}{\widetilde{u}_i^j}\\
	\leq&\sum_{i=1}^m\alpha_i\left( \frac{1}{2}\sum_{K\subset K_j^+}\int_{\partial K}\left | \mathbf{v}_i\cdot\mathbf{n} \right |(\widetilde{u}_i^j)^2+\int_{K_j^+}\epsilon (\widetilde{u}_i^j)^2\right) +\int_{K_j^+}\frac{1}{\epsilon a^\delta}\sum_{i,l=1}^m a_{il}{\widetilde{u}_l^j}{\widetilde{u}_i^j}\\
	=&s^j(u^j,u^j).
	\end{align*}
	Combining with the definition of $\left\|\cdot \right\|^2_{\widetilde{W}}$, we proved \eqref{sj_and_W}.
	
	Next, we denote $u^{j,+}=u|_{K_j^+}\in V_\text{snap}^{j,+}$. By the definition of the energy minimizing extension in \eqref{tilde}, we have
	\begin{align*}
	a^j(u^j,u^j)=&\sum_{i=1}^m\alpha_i\left( \int_{K_j^+}\left| \nabla\widetilde{u_i^j}\right| ^2+\frac{1}{H}\sum_{e\in \mathcal{E}_H^0(K_j^+)}\int_e[\widetilde{u_i^j}]^2\right) +\int_{K_j^+}\frac{1}{\epsilon a^\delta}\sum_{i,l=1}^m a_{il}\widetilde{u_l^j}\widetilde{u_i^j}\,,\\
	\leq&\sum_{i=1}^m\alpha_i\left( \int_{K_j^+}\left| \nabla{u_i^{j,+}}\right| ^2+\frac{1}{H}\sum_{e\in \mathcal{E}_H^0(K_j^+)}\int_e[{u_i^{j,+}}]^2\right) +\int_{K_j^+}\frac{1}{\epsilon a^\delta}\sum_{i,l=1}^m a_{il}{u_l^{j,+}}{u_i^{j,+}}.
	\end{align*}
	Hence,
	\begin{align*}
	\sum_j a^j(u^j,u^j)
	\leq&\sum_j \sum_{i=1}^m\alpha_i\left( \int_{K_j^+}\left| \nabla{u_i^{j,+}}\right| ^2+\frac{1}{H}\sum_{e\in \mathcal{E}_H^0(K_j^+)}\int_e[{u_i^{j,+}}]^2\right) +\sum_j \int_{K_j^+}\frac{1}{\epsilon a^\delta}\sum_{i,l=1}^m a_{il}{u_l^{j,+}}{u_i^{j,+}}\\
	\leq& \sum_{i=1}^m\alpha_i\left( \sum_jM_{K_j}\int_{K_j}\left| \nabla{u_i^{j}}\right| ^2+\frac{1}{H}\sum_{e\in \mathcal{E}_H^0}M_e\int_e[{u_i^{j}}]^2\right) +\sum_j M_{K_j}\int_{K_j}\frac{1}{\epsilon a^\delta}\sum_{i,l=1}^m a_{il}{u_l^{j}}{u_i^{j}}\\
	\leq&M\|u\|^2_\text{Energy}\,,
	\end{align*}
	and thus we have~\eqref{aj_and_Energy}.
\end{proof}

For the convergence analysis, we first examine the best approximation property. For that, we have the following:
\begin{lemma}	\label{lemma1}
	Let $u_\text{snap}$ be the snapshot solution to the equation (\ref{eq:coupled_snap}) and let $u_H$ be the multiscale solution to the equation (\ref{eq:coupled_global}). Then:
	\begin{equation}\label{eqn:best_approx}
	\left \| u_\text{snap}-u_H \right \|_{\widetilde{V}}^2\leq C\inf_{w\in V_H}\left \| u_\text{snap}-w \right \|_{\widetilde{W}}^2\,,
	\end{equation}
	where $C$ is a constant independent of $\epsilon$, $a^\delta$ and the mesh size.
\end{lemma}
\begin{proof}
	Using~\eqref{eq:coupled_snap} and~\eqref{eq:coupled_global}, and the fact that $V_H\subset V_\text{snap}$, we have:
	\begin{equation}
	a(u_\text{snap}-u_H,w)+l(u_\text{snap}-u_H,w) = F(w)-F(w) = 0\,,\quad \forall w\in V_\text{snap}\,.
	\end{equation}
	Then for all $w\in V_H$:
	\begin{align*}
	&a(u_\text{snap}-u_H,u_\text{snap}-u_H)+l(u_\text{snap}-u_H,u_\text{snap}-u_H)\\ =&a(u_\text{snap}-u_H,u_\text{snap}-w)+l(u_\text{snap}-u_H,u_\text{snap}-w)\,.
	\end{align*}
	Using Proposition \ref{a_and_V_proposition}, we have:
	\begin{align}\label{eqn:V_tilde}
	\|u_\text{snap}-u_H\|^2_{\widetilde{V}} =a(u_\text{snap}-u_H,u_\text{snap}-w)+l(u_\text{snap}-u_H,u_\text{snap}-w)\,.
	\end{align}
	To obtain~\eqref{eqn:best_approx}, noticing that $u_\text{snap}-u_H$ and $u_\text{snap}-w$ are both in $V_\text{snap}$, it amounts to show that:
	\begin{equation}\label{eqn:boundedness_tilde_V}
	a(u,w)+l(u,w)\leq C \left\| u\right\| _{\widetilde{V}}\left\| w\right\| _{\widetilde{W}}\quad\forall u,w \in V_{\text{snap}}\,.
	\end{equation}
	In fact it suffices to show that
	\begin{equation}\label{boundedness}
	a(u,w)+l(u,w)\leq \sqrt{2} \left\| u\right\| _{V}\left\| w\right\| _{\widetilde{W}}\quad\forall u,w \in V_{\text{snap}}\,,
	\end{equation}
	since it is obvious that $\|u\|_{V}\leq\|u\|_{\widetilde{V}}$.
	
	To show~\eqref{boundedness}, we first use integration by parts to obtain
	\begin{align*}
	\sum_{i=1}^m\alpha_i\left( -\int_\Omega u_i\nabla w_i\cdot\mathbf{v}_i\right) =&\sum_{i=1}^m\alpha_i\left( \int_\Omega w_i\nabla u_i\cdot\mathbf{v}_i - \sum_{\tau\in \mathcal{T}^h}\int_{\partial \tau}\mathbf{v}_i\cdot\mathbf{n}u_iw_i\right) \\
	=&\sum_{i=1}^m\alpha_i\left( \sum_{K_j}\int_{K_j} \nabla u_i\cdot\mathbf{v}_i  w_i- \sum_{K_j}\int_{\partial K_j}\mathbf{v}_i\cdot\mathbf{n}u_iw_i\right) \\
	=&-l(u,w)-\sum_{i=1}^m\alpha_i \sum_{K_j}\int_{\partial K_j}\mathbf{v}_i\cdot\mathbf{n}u_iw_i\,,
	\end{align*}
	where we have used the continuity accross fine scales $\partial\tau$, and the assumption that, in each $K_j$, $u$ satisfies the following equation
	\begin{align}\label{snapshot1_weak}
	-\int_{K_j} u_i\nabla w_i\cdot\mathbf{v}_i + \int_{\partial K_j} u_i w_i \mathbf{v}_i \cdot\mathbf{n}
	+\int_{K_j} \epsilon u_i w_i
	+\int_{K_j}\frac{1}{\epsilon a^\delta}\left( u_i-\sum_{q=1}^m \alpha_qu_q \right) w_i =0\,,   
	\end{align}
	for all $i=1,2,\cdots, m$. Here $w$ could be any function in $(V_h)^m$ restricted on $K_j$. In particular, \eqref{snapshot1_weak} works for $w \in V_{\text{snap}}^j$.
		
	Using the definition of $a(\cdot,\cdot)$ and direct calculations, we have
	\begin{align*}
	a(u,w)+l(u,w)=&\sum_{i=1}^m\alpha_i\left( -\sum_{K_j}\int_{\partial K_j}\mathbf{v}_i\cdot\mathbf{n}u_iw_i+\sum_{e\in\mathcal{E}_H^0}\int_e u_i^+[w_i]\cdot\mathbf{v}_i+\sum_{e\in{\Gamma^+}}\int_e u_iw_i\mathbf{v}_i\cdot\mathbf{n} \right) \\
	&+\sum_{i=1}^m\alpha_i\sum_{K_j}\left( \int_{K_j} \epsilon u_i w_i
	+\frac{1}{\epsilon a^\delta}\left( u_i-\sum_{q=1}^m \alpha_qu_q \right) w_i\right) \\
	=&\sum_{i=1}^m\alpha_i\left( -\sum_{e\in\mathcal{E}_H^0}\int_e w_i^-[u_i]\cdot\mathbf{v}_i-\sum_{e\in{\Gamma^-}}\int_e u_iw_i\mathbf{v}_i\cdot\mathbf{n}\right) \,.
	\end{align*}
	Then, applying the Cauchy-Schwarz inequality, we have 
	\begin{align}\label{eqn:conv_CS}
	a(u,w)+l(u,w)&\leq\left (\sum_{i=1}^m\alpha_i\left( \sum_{e\in\mathcal{E}_H^0}\int_e \left | \mathbf{v}_i\cdot\mathbf{n} \right |[u_i]^2+\sum_{e\in{\Gamma^-}}\int_e \left | \mathbf{v}_i\cdot\mathbf{n} \right |u_i^{2}  \right)   \right )^{1/2}\nonumber\\
	&\quad\left (\sum_{i=1}^m\alpha_i\left( \sum_{e\in\mathcal{E}_H^0}\int_e \left | \mathbf{v}_i\cdot\mathbf{n} \right |{w_i^{-}}^2+\sum_{e\in{\Gamma^-}}\int_e \left | \mathbf{v}_i\cdot\mathbf{n} \right |w_i^{2}\right)  \right )^{1/2}.
	\end{align}
	The two terms on the right hand side are taken care of separately. To handle the first term, recall the definition of $V$-norm in equation \eqref{def:V_norm}, we have
	\begin{align}\label{eqn:conv_term1}
	\sum_{i=1}^m\alpha_i\left( \sum_{e\in\mathcal{E}_H^0}\int_e \left | \mathbf{v}_i\cdot\mathbf{n} \right |[u_i]^2+\sum_{e\in{\Gamma^-}}\int_e \left | \mathbf{v}_i\cdot\mathbf{n} \right |u_i^{2}  \right) \leq \sum_{i=1}^m\alpha_i \sum_{e\in \mathcal{E}_H}\int_e\left | \mathbf{v}_i\cdot\mathbf{n} \right |[u_i]^2=\sqrt{2}\left \| u \right \|_{V}^2\,.
	\end{align}
	
	And to compute the second term, we notice that
	\begin{align*}
	0=&\sum_{i=1}^m\alpha_i\sum_{K_j}\left( -\int_{K_j} w_i\nabla w_i\cdot\mathbf{v}_i + \int_{\partial K_j} w_i^2 \mathbf{v}_i \cdot\mathbf{n}\right) +l(w,w)\\
	=&\sum_{i=1}^m\alpha_i\sum_{K_j}\left( -\frac{1}{2}\int_{\partial K_j} w_i^2 \mathbf{v}_i \cdot\mathbf{n}+ \int_{\partial K_j} w_i^2 \mathbf{v}_i \cdot\mathbf{n}\right) +l(w,w)\\
	=&\frac{1}{2}\sum_{i=1}^m\alpha_i\sum_{K_j}\int_{\partial K_j}\mathbf{v}_i\cdot\mathbf{n}w_i^2 +l(w,w)\\
	=&\frac{1}{2}\sum_{i=1}^m\alpha_i\left ( 
	-\sum_{e\in \Gamma^-}\int_{e}\left |\mathbf{v}_i\cdot\mathbf{n}  \right |w_i^2 + 
	\sum_{e\in \Gamma^+}\int_{e}\left |\mathbf{v}_i\cdot\mathbf{n}  \right |w_i^2+ \sum_{e\in \mathcal{E}_H^0}\int_{e}\left |\mathbf{v}_i\cdot\mathbf{n}  \right |\left ({w_i^{+}}^2-{w_i^{-}}^2  \right ) \right )+l(w,w)\,,
	\end{align*}
	which in turn gives
	\begin{align}\label{eqn:conv_term2}
	&\sum_{i=1}^m\alpha_i\left( \sum_{e\in\mathcal{E}_H^0}\int_e \left | \mathbf{v}_i\cdot\mathbf{n} \right |{w_i^{-}}^2+\sum_{e\in{\Gamma^-}}\int_e \left | \mathbf{v}_i\cdot\mathbf{n} \right |w_i^{2} \right)  \nonumber\\
	=&\frac{1}{2}\sum_{i=1}^m\alpha_i\left ( 
	\sum_{e\in \Gamma^-}\int_{e}\left |\mathbf{v}_i\cdot\mathbf{n}  \right |w_i^2 + 
	\sum_{e\in \Gamma^+}\int_{e}\left |\mathbf{v}_i\cdot\mathbf{n}  \right |w_i^2+ \sum_{e\in \mathcal{E}_H^0}\int_{e}\left |\mathbf{v}_i\cdot\mathbf{n}  \right |\left (w_i^{+2}+w_i^{-2}  \right ) \right )+l(w,w)\nonumber\\
	=&\frac{1}{2}\sum_{i=1}^m\alpha_i\sum_{K_j}\int_{\partial K_j}\left| \mathbf{v}_i\cdot\mathbf{n}\right| w_i^2+l(w,w)\nonumber\\
	=&\left\| w\right\| _{\widetilde{W}}^2\,.
	\end{align}
	Plug~\eqref{eqn:conv_term1} and~\eqref{eqn:conv_term2} into~\eqref{eqn:conv_CS}, we have proved the desired boundedness condition \eqref{boundedness} which concludes the proof of~\eqref{eqn:best_approx}.
\end{proof}

Now, we are ready to prove our main convergence result in this section.

\begin{theorem}\label{final}
	Let $u_\text{snap}$ be the snapshot solution to problem (\ref{eq:coupled_snap}) and let $u_H$ be the multiscale solution to problem (\ref{eq:coupled_global}). Then:
	$$ \left \| u_\text{snap}-u_H \right \|_{\widetilde{V}}^2\leq \frac{CM}{\Lambda_*}\left\|u_\text{snap} \right\|_\text{Energy}^2,$$		
	where
	$\Lambda_*=\min_j \lambda_{j,L_j+1}$, $C$ is the same constant from Lemma \ref{lemma1}, and $M$ is the same constant from Proposition \ref{norm_connection}.
\end{theorem}
\begin{proof}
	We first denote
	\[
	u_\text{snap}=\sum_{j}u_\text{snap}|_{K_j}=\sum_{j}u_\text{snap}^j=\sum_{j,l}d_{j,l}\phi_{l}^j\,,
	\]
	where $\phi_l^j$ is the $l$-th multiscale basis function for the coarse element $K_j$~\eqref{eqn:GEP}. Note that $\text{span}\{\phi_l^j\}$ covers the entire snapshot space. We then define a projection of $u_\text{snap}^j$ into $V_H^j$, as well as a projection of $u_\text{snap}$ into $V_H$:
	\[
	P^j(u_\text{snap}^j)=\sum_{l\leq L_j}d_{j,l}\phi_{l}^j\,,\quad P(u_\text{snap})=\sum_j\sum_{l\leq L_j}d_{j,l}\phi_{l}^j.
	\]
	It is easy to see that $P^j(u_\text{snap}^j)=P(u_\text{snap})|_{K_j}$.
	Combining with Proposition \ref{norm_connection}, we have
	\begin{align*}
	\inf_{w\in V_H}\left \| u_\text{snap}-w \right \|_{\widetilde{W}}^2&\leq\left \| u_\text{snap}-P(u_\text{snap}) \right \|_{\widetilde{W}}^2\\
	&\leq\sum_{j}s^j\left( u_{\text{snap}}^j-P^j(u_{\text{snap}}^j),u_{\text{snap}}^j-P^j(u_{\text{snap}}^j)\right) \\
	&\leq\sum_{j}\frac{1}{\lambda_{L_j+1}^{(j)}}a^j\left( u_{\text{snap},j},u_{\text{snap},j}\right)\\
	&\leq\frac{1}{\Lambda_*}\sum_{j}a^j\left( u_{\text{snap}}^j,u_{\text{snap}}^j\right)\\
	&\leq\frac{M}{\Lambda_*}\left\|u_\text{snap} \right\|_\text{Energy}^2. 
	\end{align*}
	Combining with Lemma \ref{lemma1}, we proved the theorem.
\end{proof}

In the above theorem, we estimate the error between the snapshot solution $u_{\text{snap}}$ and the multiscale solution $u_H$.
We see that the error is inversely proportional to the eigenvalues. This shows that the multiscale space gives a good approximation property 
in the snapshot space. In our analysis, we assume that the snapshot functions satisfy the PDE in the strong sense, that is, (\ref{snapshot1_weak}).
On the other hand, there is an error between the snapshot solution $u_{\text{snap}}$ and the fine scale solution $u_h$ if we use Algorithm \textproc{RanLocal} in Section \ref{sec:basis}.
This amounts to an irreducible error, and the analysis of this is beyond the scope of this paper.

We should emphasize that the difficulty brought by small $\delta$ is encoded in the quality of $\Lambda_\ast$ and thus is not explicitly expressed in the error analysis.

\subsection{Small $\epsilon$ regime}\label{sec:small}
An important property the algorithm satisfies is that it is robust with respect to the parameters. In the limiting regime of $\epsilon\to0$, $\Lambda_*$ has a positive lower bound, and this serves as the stability argument that allows the algorithm to be effective across regimes. In particular we will show:

\begin{theorem}\label{thm:expansion}
	Denote $\lambda_{j,k}$ the $k$-th eigenvalue of the GEP defined in~\eqref{eqn:GEP} for coarse element $K_j$. It has a asymptotic limit in the zero limit of $\epsilon$, meaning there is a constant $\lambda^0_{j,k}$ so that:
	\[
	\left| \lambda_{j,k}-\lambda^0_{j,k}\right| =\mathcal{O}(\epsilon)\,.
	\]	
\end{theorem}

This theorem, when combined with our main Theorem~\ref{final}, indicates that the error bound, which is controlled by $\Lambda_\ast=\frac{1}{\text{min}_j\{\lambda_{j,L_{j+1}}\}}$, will not grow in $\epsilon$ and thus the error is uniformly bounded.

To show the theorem, we first start with a lemma.

\begin{lemma}\label{lemma2}
	For every coarse element $K_j$, we have
	\[
	A^{j} = A^{j,0} + \mathcal{O}(\epsilon)\,,\quad \text{and}\quad S^{j} = S^{j,0} + \mathcal{O}(\epsilon)\,,
	\]
	where entries in $A^{j,0}$ and $S^{j,0}$ are defined by: 
	\begin{equation}\label{eqn:A^0}
	A_{pq}^{j,0}=\sum_{i=1}^m\alpha_i\left( \int_{K_j^+}\nabla\widetilde{\psi}_{p,i}^{j,0}\cdot\nabla\widetilde{\psi}_{q,i}^{j,0}+\frac{1}{H}\sum_{e\in \mathcal{E}_H^0(K_j^+)}\int_e[\widetilde{\psi}_{p,i}^{j,0}][\widetilde{\psi}_{q,i}^{j,0}]\right)\,,
	\end{equation}
	and
	\begin{equation}\label{eqn:S^0}
	S^{j,0}_{pq} = \sum_{i=1}^m\alpha_i\left( \frac{1}{2}\sum_{K\subset K_j^+}\int_{\partial K}\left | \mathbf{v}_i\cdot\mathbf{n} \right |\widetilde{\psi}_{p,i}^{j,0}\widetilde{\psi}_{q,i}^{j,0}\right)\,,
	\end{equation}
	where $V_\text{snap}^j=\text{span}\{\psi_p^j\}$, and $\widetilde{\psi}_{p}^{j}$ is the basis functions' energy minimizing extension. We further denote $\widetilde{\psi}_{p}^{j,0}$ the leading order asymptotic expansion of $\widetilde{\psi}_{p}^{j}$.
\end{lemma}

\begin{proof}
	To proceed we notice that $\widetilde{\psi}_{p}^j\in V_\text{snap}^{j,+}$ can be written as the sum of some ${\psi}^s$'s, where ${\psi}^s\in V_\text{snap}^s$ and $K_s\subset K_j^+$. Recall the assumption of equation \eqref{snapshot1_weak}. Then in each $K_s$ and for all $i=1,2,...,m$, we have:
	
	\begin{align}
	-\int_{K_j} \widetilde{\psi}_{p,i}^j\nabla w_i\cdot\mathbf{v}_i + \int_{\partial K_j} \widetilde{\psi}_{p,i}^j w_i \mathbf{v}_i \cdot\mathbf{n}
	+\int_{K_j} \epsilon \widetilde{\psi}_{p,i}^j w_i
	+\int_{K_j}\frac{1}{\epsilon a^\delta}\left( \widetilde{\psi}_{p,i}^j-\sum_{l=1}^m \alpha_l\widetilde{\psi}_{p,l}^j \right) w_i =0.
	\end{align}
	Here $w$ could be any function in $(V_h)^m$ restricted on $K_j$.
	
	Therefore we have:
	\begin{align}\label{eqn:snapshot3}
	\sum_{i=1}^m\alpha_i\left( -\int_{K_j} \widetilde{\psi}_{p,i}^j\nabla w_i\cdot\mathbf{v}_i + \int_{\partial K_j} \widetilde{\psi}_{p,i}^j w_i \mathbf{v}_i \cdot\mathbf{n}
	+\int_{K_j} \epsilon \widetilde{\psi}_{p,i}^j w_i
	\right)+\int_{K_j}\frac{1}{\epsilon a^\delta}\sum_{i,l=1}^m a_{il}\widetilde{\psi}_{p,l}^{j,0}w_i=0.
	\end{align}
	
	Take the asymptotic expansion for $\widetilde{\psi}_{p}^j$:
	\begin{equation}\label{expansion}
	\widetilde{\psi}_{p}^j=\widetilde{\psi}_{p}^{j,0}+\epsilon\widetilde{\psi}_{p}^{j,1}+\mathcal{O}(\epsilon^2)\,,
	\end{equation}
	set $w=\widetilde{\psi}_{p}^{j,0}$, and plug them back into~\eqref{eqn:snapshot3}. We have, in the leading order of $\frac{1}{\epsilon}$:
	\[
	\sum_{i,l=1}^m a_{il}\widetilde{\psi}_{p,l}^{j,0}\widetilde{\psi}_{p,i}^{j,0}=0\,,
	\]
	meaning $\widetilde{\psi}_{p}^{j,0}$ is isotropic in each $K_s$ due to Proposition \ref{prop:semi-definite}. Therefore $\widetilde{\psi}_{p}^{j,0}$ is isotropic. The same analysis is applied to $\widetilde{\psi}_{q}^{j,0}$.
	
	Recall the definition of $A^j$ and $a^n$ in~\eqref{spectral}, we have:
	\begin{align*}
	a^j(\psi_{p}^j,\psi_{q}^j)=&\sum_{i=1}^m\alpha_i\left( \int_{K_j^+}\nabla\widetilde{\psi}_{p,i}^{j}\cdot\nabla\widetilde{\psi}_{q,i}^{j}+\frac{1}{H}\sum_{e\in \mathcal{E}_H^0(K_j^+)}\int_e[\widetilde{\psi}_{p,i}^{j}][\widetilde{\psi}_{q,i}^{j}]\right) +\int_{K_j^+}\frac{1}{\epsilon a^\delta}\sum_{i,l=1}^m a_{il}\widetilde{\psi}_{p,l}^{j}\widetilde{\psi}_{q,i}^{j}\\
	=&\frac{1}{\epsilon}\left[\int_{K_j^+}\frac{1}{a^\delta}\sum_{i,l=1}^m a_{il}\widetilde{\psi}_{p,l}^{j,0}\widetilde{\psi}_{q,i}^{j,0}\right]\\
	&+1\left[\sum_{i=1}^m\alpha_i\left( \int_{K_j^+}\nabla\widetilde{\psi}_{p,i}^{j,0}\cdot\nabla\widetilde{\psi}_{q,i}^{j,0}+\frac{1}{H}\sum_{e\in \mathcal{E}_H^0(K_j^+)}\int_e[\widetilde{\psi}_{p,i}^{j,0}][\widetilde{\psi}_{q,i}^{j,0}]\right) \right.\\
	&+\left.\int_{K_j^+}\frac{1}{ a^\delta}\sum_{i,l=1}^m a_{il}\widetilde{\psi}_{p,l}^{j,0}\widetilde{\psi}_{q,i}^{j,1}+\int_{K_j^+}\frac{1}{ a^\delta}\sum_{i,l=1}^m a_{il}\widetilde{\psi}_{p,l}^{j,1}\widetilde{\psi}_{q,i}^{j,0}\right]\\
	&+\mathcal{O}(\epsilon)\,.
	\end{align*}
	
	Due to Proposition~\ref{prop:semi-definite}, we have:
	\begin{equation*}
	\int_{K_j^+}\frac{1}{a^\delta}\sum_{i,l=1}^m a_{il}\widetilde{\psi}_{p,l}^{j,0}\widetilde{\psi}_{q,i}^{j,0}=\int_{K_j^+}\frac{1}{a^\delta}\sum_{i,l=1}^m a_{il}\widetilde{\psi}_{p,l}^{j,0}\widetilde{\psi}_{q,i}^{j,1}=\int_{K_j^+}\frac{1}{a^\delta}\sum_{i,l=1}^m a_{il}\widetilde{\psi}_{p,l}^{j,1}\widetilde{\psi}_{q,i}^{j,0}=0\,,
	\end{equation*}
	and thus
	\begin{align*}
	a^j(\psi_{p}^j,\psi_{q}^j)&=\sum_{i=1}^m\alpha_i\left( \int_{K_j^+}\nabla\widetilde{\psi}_{p,i}^{j,0}\cdot\nabla\widetilde{\psi}_{q,i}^{j,0}+\frac{1}{H}\sum_{e\in \mathcal{E}_H^0(K_j^+)}\int_e[\widetilde{\psi}_{p,i}^{j,0}][\widetilde{\psi}_{q,i}^{j,0}]\right)+\mathcal{O}(\epsilon)\\
	&\to\sum_{i=1}^m\alpha_i\left( \int_{K_j^+}\nabla\widetilde{\psi}_{p,i}^{j,0}\cdot\nabla\widetilde{\psi}_{q,i}^{j,0}+\frac{1}{H}\sum_{e\in \mathcal{E}_H^0(K_j^+)}\int_e[\widetilde{\psi}_{p,i}^{j,0}][\widetilde{\psi}_{q,i}^{j,0}]\right)\,.
	\end{align*}
	
	The proof for $s^j$ is the same and is omitted here.
\end{proof}

Theorem \ref{thm:expansion} is straightforward consequence of the following perturbation theorem:
\begin{proof}[Proof for Theorem~\ref{thm:expansion}]
	According to Lemma~\ref{lemma2}, $A^j$ and $S^j$ have expansions $A^j=A^{j,0}+\mathcal{O}(\epsilon)=:A^{j,0}+\epsilon A^{j,1}$ and $S^j=S^{j,0}+\mathcal{O}(\epsilon)=:S^{j,0}+\epsilon S^{j,1}$. We also define $\mathbf{x}_k^{j,0}$ as the $k$-th generalized eigenvector of the two matrices $A^{j,0}$ and $S^{j,0}$, i.e.
	\begin{equation}\label{GEP_0}
	A^{j,0}\mathbf{x}_k^{j,0}=\lambda_{j,k}^0S^{j,0}\mathbf{x}_k^{j,0}.
	\end{equation}
	Using Absolute Weyl theorem for generalized eigenvalue problems in \cite{nakatsukasa2010absolute}, when $\epsilon$ is small enough such that $\epsilon\left\|S^{j,1} \right\|_2<\lambda_{\min}(S^{j,0}) $, we have:
	\begin{align*}
	\left| \lambda_{j,k}-\lambda^0_{j,k}\right| \leq&\frac{\left\| \epsilon A^{j,1}\right\|_2}{\lambda_{\min}(S^{j,0})}+\frac{\left\| A^{j,0}\right\|_2+\left\| \epsilon A^{j,1}\right\|_2}{\lambda_{\min}(S^{j,0})(\lambda_{\min}(S^{j,0})-\left\| \epsilon S^{j,1}\right\|_2)}\left\| \epsilon S^{j,1}\right\|_2\\
	=&\frac{\epsilon\left\| A^{j,1}\right\|_2}{\lambda_{\min}(S^{j,0})}+\frac{\left\| A^{j,0}\right\|_2+\epsilon\left\| A^{j,1}\right\|_2}{\lambda_{\min}(S^{j,0})(\lambda_{\min}(S^{j,0})-\epsilon\left\| S^{j,1}\right\|_2)}\epsilon\left\| S^{j,1}\right\|_2\\
	=&\frac{\epsilon \mathcal{O}(1)}{\mathcal{O}(1)}+\frac{\mathcal{O}(1)+\epsilon \mathcal{O}(1)}{\mathcal{O}(1)(\mathcal{O}(1)-\epsilon\mathcal{O}(1))}\epsilon\mathcal{O}(1)\\
	=&\mathcal{O}(\epsilon)
	\end{align*}
	where $\left\| \cdot\right\| _2$ is the spectral norm of a matrix.
\end{proof}

According to the formula for $A^{j,0}$ and $S^{j,0}$ in \eqref{eqn:A^0} and~\eqref{eqn:S^0}, the eigenvalues are positive except that the smallest one is 0 with constant as corresponding eigenvector. So $\Lambda_*$ has positive limit in the limiting regime of $\epsilon\to0$.
%
%
%
%
%
%
%

\section{Numerical results}
We take boundary condition $g(x,\mathbf{v})=\cos(2\pi(x_1+x_2))+1$. And we set $m=6$, and use Gaussian quadrature rule to define
$\{(\mathbf{v}_i,\alpha_i), i = 1,. . . ,m\}$. As for $a^\delta$, we give two examples. In the first example, we will choose $a^\delta$ to be based on a high contrast media $\kappa$, shown in Figure \ref{coefficient}(Left), and choose $a^\delta$ to be an oscillatory function for the second example used in \cite{ll16,ehw00,hw99}, shown in Figure \ref{coefficient}(Right), with expression
\begin{equation*}
a^\delta=\frac{2+1.8\sin(10\pi x_1)}{2+1.8\cos(10\pi x_2)}+\frac{2+\sin(10\pi x_2)}{2+1.8\sin(10\pi x_1)}.
\end{equation*}

\begin{figure}[ht]
	\centering
	\begin{minipage}[t]{0.5\textwidth}
		\centering
		\includegraphics[width=3in]{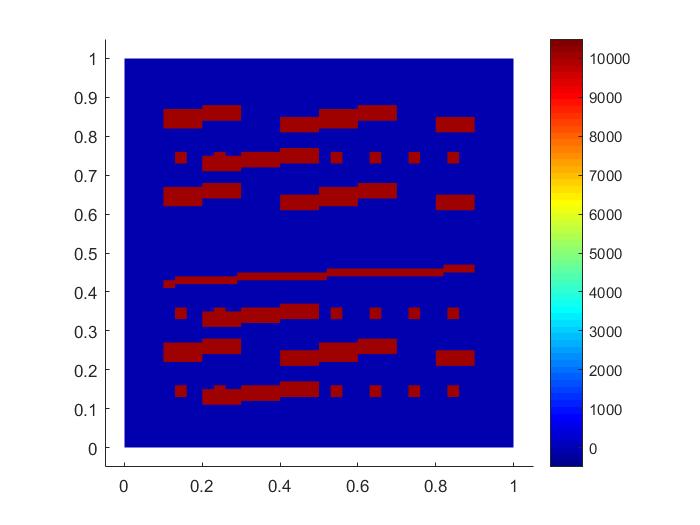}
	\end{minipage}%
	\begin{minipage}[t]{0.5\textwidth}
		\centering
		\includegraphics[width=3in]{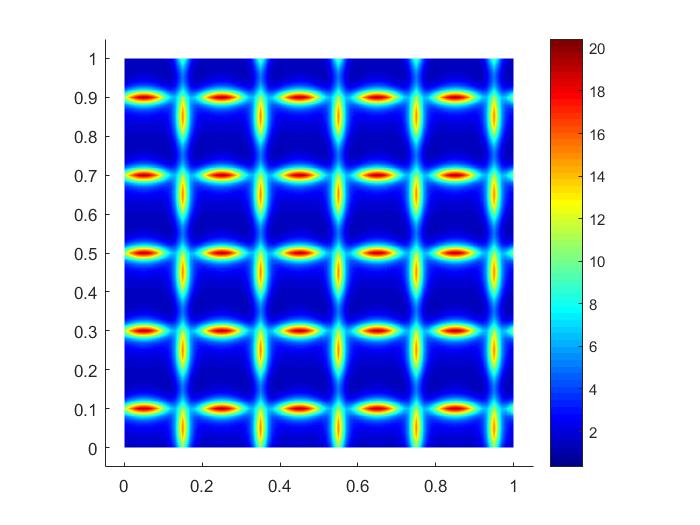}
	\end{minipage}
	\caption{Left: $\kappa$ for Example 1. Right: $a^\delta$ for Example 2.}
	\label{coefficient}
\end{figure}

The space domain $\Omega$ is taken as
the unit square $[0,1]\times[0,1]$ and is divided into $10\times 10$ coarse blocks consisting of uniform squares.
Each coarse element is then divided into $10\times 10$ fine elements consisting of uniform squares. That is, $\Omega$ is partitioned by $100\times 100$ square fine elements. And we use oversampling technique in equation \eqref{def:V_snap_j_oversampling}-\eqref{snapshot2_oversampling} to obtain the snapshot space. We define an oversampling region $K_j^+$ by enlarging $K_j$ by one coarse grid layer.

To compare the accuracy, we will use the following error quantities:
\begin{equation*}
e_1=\left( \frac{\sum_{i=1}^m\alpha_i\int_{\Omega}|u_{h_i}-u_{H,i}|^2}{\sum_{i=1}^m\alpha_i\int_{\Omega}|u_{h_i}|^2}\right) ^{1/2},\quad
e_2=\left( \frac{\int_{\Omega}|\overline{u_{h}}-\overline{u_{H}}|^2}{\int_{\Omega}|\overline{u_{h}}|^2}\right) ^{1/2},
\end{equation*}
where $\overline{u}$ is defined as $\overline{u}=\sum_{i=1}^m\alpha_i u_i.$

For Example 1, we first fix $a^\delta=\kappa^4$ and give the error tables for Knudsen number $\epsilon=10^{-1},10^{-2},10^{-3}$, respectively. And $L$ is the number of multiscale basis chosen from each coarse element, and snapshot ratio is define by
$$\text{snapshot ratio}=\frac{\text{dim}(V_H)}{\text{dim}(V_\text{snap})}.$$

From Table \ref{ex1 error table epsilon-1 kappa4}, \ref{ex1 error table epsilon-2 kappa4}, \ref{ex1 error table epsilon-3 kappa4}, we can see this framework works for all Knudsen number $\epsilon$, which verifies our proved conclusion. In addition, we see
clearly the reduction of error when more basis functions are used, and the reduction of error is more
rapid when fewer basis functions are used. We also observe that the method gives reasonable error
levels with small snapshot ratios. On the other hand, Figures \ref{ex1 solution} show the fine and multiscale
solutions with $\epsilon=10^{-2}$ and $L=5$. From these figures, we observe very good agreements between the fine-scale and multiscale solutions

Next, we fix $\epsilon=10^{-2}$ and change the high contrast value of $a^\delta$. We set $a^\delta=\kappa^2,\;\kappa^4,\;\kappa^6$, respectively. From Table \ref{ex1 error table change high contrast}, we can see that contrast values do not affect the error.

\begin{table}[ht]
	\begin{minipage}[t]{0.5\textwidth}
		\centering
		\begin{tabular}{|c|c|c|c|}
			\hline
			\hline
			$L$  & snapshot ratio & $e_1$ & $e_2$ \\
			\hline
			1  & 0.79\%  &  19.64\% &   9.41\% \\
			\hline
			2 & 1.59\% & 17.68\%  &  8.53\% \\
			\hline
			3  & 2.38\%  &  14.41\% &  7.40\% \\
			\hline
			5 & 3.97\%  & 8.11\%  & 4.92\%  \\
			\hline
			7 & 5.56\%  & 6.16\%  & 3.62\%  \\
			\hline
			10 & 7.94\%  & 3.44\%  &  1.62\% \\
			\hline
			15 &  11.90\% & 2.24\% & 1.04\%  \\
			\hline
			20 &  15.87\% &  1.64\% & 0.68\%  \\	
			\hline
		\end{tabular}
		\caption{Errors for Example 1 with $\epsilon=10^{-1}$ and $a^\delta=\kappa^4$.}
		\label{ex1 error table epsilon-1 kappa4}
	\end{minipage}%
	\begin{minipage}[t]{0.5\textwidth}
		\centering
		\begin{tabular}{|c|c|c|c|}
			\hline
			\hline
			$L$  & snapshot ratio & $e_1$ & $e_2$ \\
			\hline
			1  & 0.79\%  &  12.05\% &   11.69\% \\
			\hline
			2 & 1.59\% & 15.35\%  &  15.17\% \\
			\hline
			3  & 2.38\%  &  3.73\% &  3.44\% \\
			\hline
			5 & 3.97\%  & 2.90\%  & 2.64\%  \\
			\hline
			7 & 5.56\%  & 2.61\%  & 2.41\%  \\
			\hline
			10 & 7.94\%  & 1.86\%  &  1.67\% \\
			\hline
			15 &  11.90\% & 1.20\% & 0.98\%  \\
			\hline
			20 &  15.87\% &  1.04\% & 0.83\%  \\	
			\hline
		\end{tabular}
		\caption{Errors for Example 1 with $\epsilon=10^{-2}$ and $a^\delta=\kappa^4$.}
		\label{ex1 error table epsilon-2 kappa4}
	\end{minipage}
\end{table}

\begin{table}[ht]
	\begin{minipage}[t]{0.5\textwidth}
		\centering
		\begin{tabular}{|c|c|c|c|}
			\hline
			\hline
			$L$  & snapshot ratio & $e_1$ & $e_2$ \\
			\hline
			1  & 0.79\%  &  12.80\% &   12.80\% \\
			\hline
			2 & 1.59\% & 26.43\%  &  26.42\% \\
			\hline
			3  & 2.38\%  &  17.86\% &  17.85\% \\
			\hline
			5 & 3.97\%  & 4.45\%  & 4.43\%  \\
			\hline
			7 & 5.56\%  & 3.60\%  & 3.59\%  \\
			\hline
			10 & 7.94\%  & 3.55\%  &  3.53\% \\
			\hline
			15 &  11.90\% & 3.20\% & 3.18\%  \\
			\hline
			20 &  15.87\% &  3.19\% & 3.17\%  \\	
			\hline
		\end{tabular}
		\caption{Errors for Example 1 with $\epsilon=10^{-3}$ and $a^\delta=\kappa^4$.}
		\label{ex1 error table epsilon-3 kappa4}
	\end{minipage}%
	\begin{minipage}[t]{0.5\textwidth}
		\centering
		\begin{tabular}{|c|c|c|c|}
			\hline
			\hline
			$L$  & $\kappa^2$ & $\kappa^4$ & $\kappa^6$ \\
			\hline
			1  & 11.70\%  &  11.69\% &   11.68\% \\
			\hline
			2 & 15.19\% & 15.17\%  &  15.17\% \\
			\hline
			3  & 3.41\%  &  3.44\% &  3.45\% \\
			\hline
			5 & 2.60\%  & 2.64\%  & 2.64\%  \\
			\hline
			7 & 2.37\%  & 2.41\%  & 2.41\%  \\
			\hline
			10 & 1.64\%  & 1.67\%  &  1.67\% \\
			\hline
			15 &  0.96\% & 0.98\% & 0.98\%  \\
			\hline
			20 &  0.81\% &  0.83\% & 0.83\%  \\	
			\hline
		\end{tabular}
		\caption{$e_2$ for Example 1 with different high contrast value of $a^\delta$.}
		\label{ex1 error table change high contrast}
	\end{minipage}
\end{table}

\begin{figure}[ht]
	\centering
	\begin{minipage}[t]{0.45\textwidth}
		\centering
		\includegraphics[width=3in]{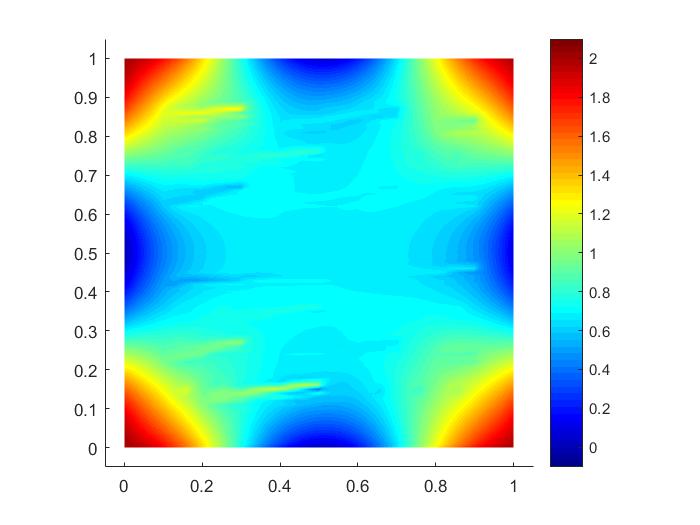}
	\end{minipage}%
	\begin{minipage}[t]{0.45\textwidth}
		\centering
		\includegraphics[width=3in]{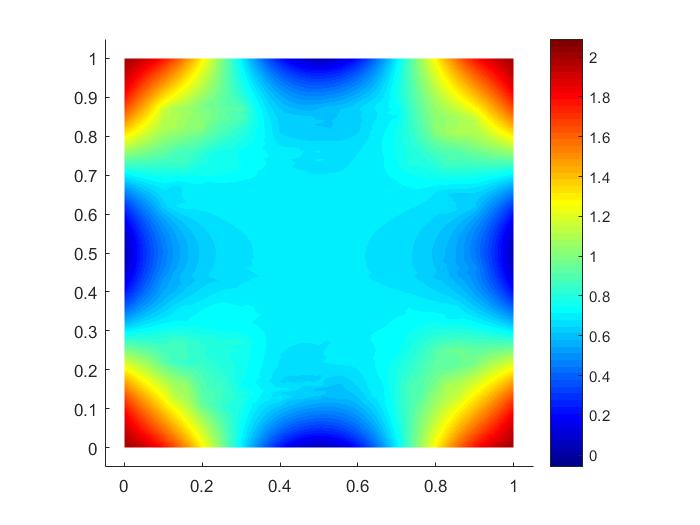}
	\end{minipage}\\%
	\begin{minipage}[t]{0.45\textwidth}
		\centering
		\includegraphics[width=3in]{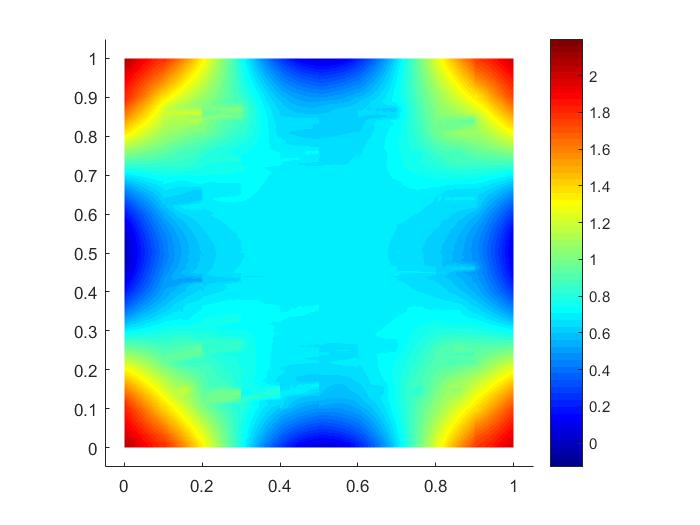}
	\end{minipage}%
	\begin{minipage}[t]{0.45\textwidth}
		\centering
		\includegraphics[width=3in]{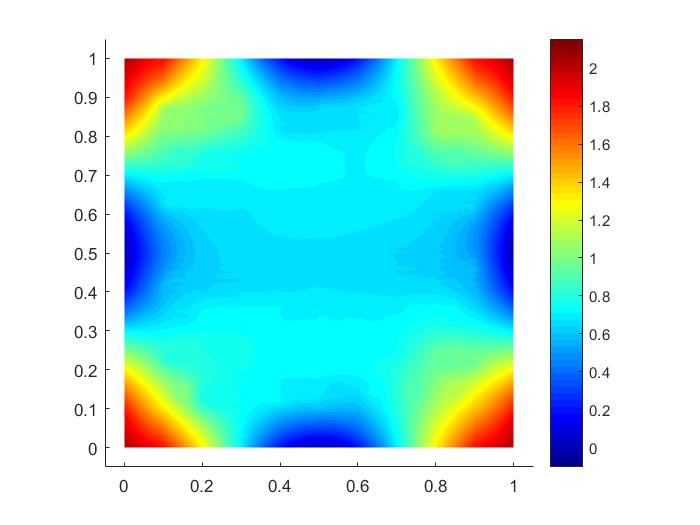}
	\end{minipage}%
	\caption{Fine solution and multiscale solution for Example 1. Top-Left: $u_{h,1}$. Top-Right: $\overline{u_h}$. Bottom-Left: $u_{H,1}$. Bottom-Right: $\overline{u_H}$. }
	\label{ex1 solution}
\end{figure}

For Example 2, we give the error tables for $\epsilon=5\times10^{-2},5\times10^{-3},5\times10^{-4}$, respectively. We present the errors for using various choices of number of basis functions in Table \ref{ex2 error table epsilon5-2}, \ref{ex2 error table epsilon5-3}, \ref{ex2 error table epsilon5-4}. We clearly see that, with a very small snapshot ratio, our method is able to
obtain solutions with very good accuracy. Furthermore, we observe a faster decay of the error when
smaller number of basis functions are used. In Figures \ref{ex2 solution}, we present the fine and multiscale
solutions with $\epsilon=5\times10^{-3}$ and $L=5$. We observe very good agreement of both solutions.

\begin{table}[ht]
	\begin{minipage}[t]{0.5\textwidth}
		\centering
		\begin{tabular}{|c|c|c|c|}
			\hline
			\hline
			$L$  & snapshot ratio & $e_1$ & $e_2$ \\
			\hline
			1  & 0.79\%  &  22.70\% &   9.73\% \\
			\hline
			2 & 1.59\%  & 20.36\%  &  8.43\% \\
			\hline
			3  & 2.38\%  &  16.97\% &  8.13\% \\
			\hline
			5 & 3.97\%  & 11.94\%  & 6.86\%  \\
			\hline
			7 & 5.56\%  & 8.09\%  & 4.64\%  \\
			\hline
			10 & 7.94\%  & 4.70\%  &  1.99\% \\
			\hline
			15 &  11.90\% & 2.48\% & 1.22\%  \\
			\hline
			20 &  15.87\% &  1.86\% & 0.91\%  \\	
			\hline
		\end{tabular}
		\caption{Errors for Example 2 with $\epsilon=5\times10^{-2}$.}
		\label{ex2 error table epsilon5-2}
	\end{minipage}%
	\begin{minipage}[t]{0.5\textwidth}
		\centering
		\begin{tabular}{|c|c|c|c|}
			\hline
			\hline
			$L$  & snapshot ratio & $e_1$ & $e_2$ \\
			\hline
			1  & 0.79\%  &  12.76\% &   11.98\% \\
			\hline
			2 & 1.59\%  & 11.02\%  &  10.64\% \\
			\hline
			3  & 2.38\%  &  3.40\% &  2.97\% \\
			\hline
			5 & 3.97\%  & 2.04\%  & 1.67\%  \\
			\hline
			7 & 5.56\%  & 1.77\%  & 1.43\%  \\
			\hline
			10 & 7.94\%  & 1.50\%  &  1.21\% \\
			\hline
			15 &  11.90\% & 1.38\% & 1.15\%  \\
			\hline
			20 &  15.87\% &  1.17\% & 0.95\%  \\	
			\hline
		\end{tabular}
		\caption{Errors for Example 2 with $\epsilon=5\times10^{-3}$.}
		\label{ex2 error table epsilon5-3}
	\end{minipage}
\end{table}

\begin{table}[ht]
	\centering 
	\begin{tabular}{|c|c|c|c|}
		\hline
		\hline
		$L$  & snapshot ratio & $e_1$ & $e_2$ \\
		\hline
		1  & 0.79\%  &  14.12\% &   14.11\% \\
		\hline
		2 & 1.59\%  & 20.87\%  &  20.86\% \\
		\hline
		3  & 2.38\%  &  11.69\% &  11.69\% \\
		\hline
		5 & 3.97\%  & 2.95\%  & 2.95\%  \\
		\hline
		7 & 5.56\%  & 2.71\%  & 2.71\%  \\
		\hline
		10 & 7.94\%  & 2.71\%  &  2.71\% \\
		\hline
		15 &  11.90\% & 2.88\% & 2.88\%  \\
		\hline
		20 &  15.87\% &  2.93\% & 2.92\%  \\	
		\hline
	\end{tabular}
	\caption{Errors for Example 2 with $\epsilon=5\times10^{-4}$.}
	\label{ex2 error table epsilon5-4}
\end{table}

\begin{figure}[ht]
	\centering
	\begin{minipage}[t]{0.45\textwidth}
		\centering
		\includegraphics[width=3in]{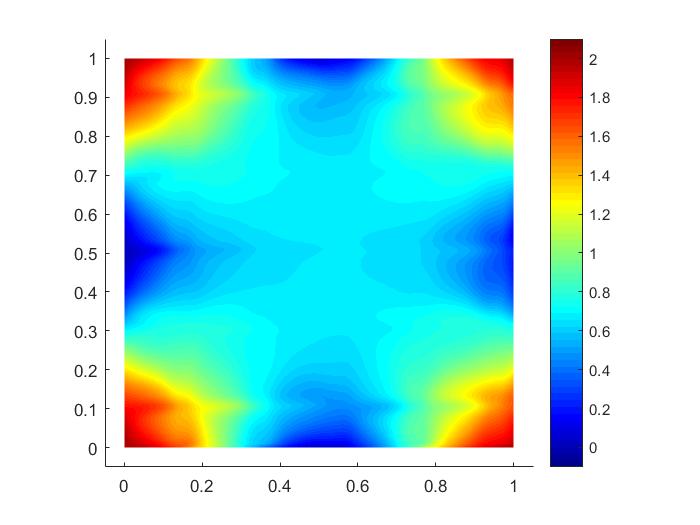}
	\end{minipage}%
	\begin{minipage}[t]{0.45\textwidth}
		\centering
		\includegraphics[width=3in]{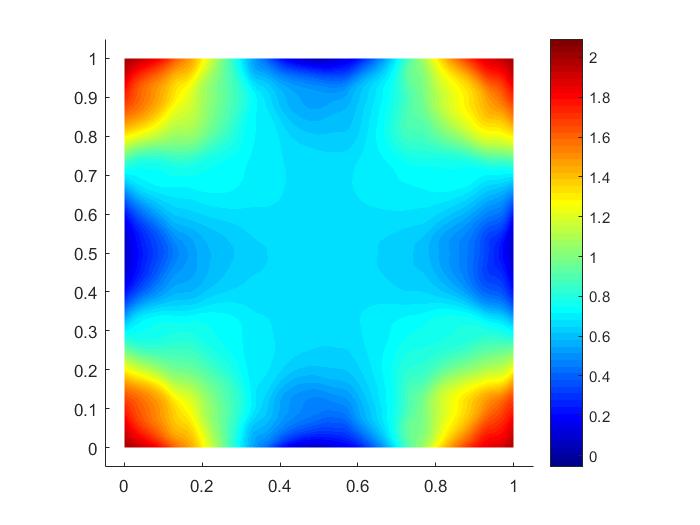}
	\end{minipage}\\%
	\begin{minipage}[t]{0.45\textwidth}
		\centering
		\includegraphics[width=3in]{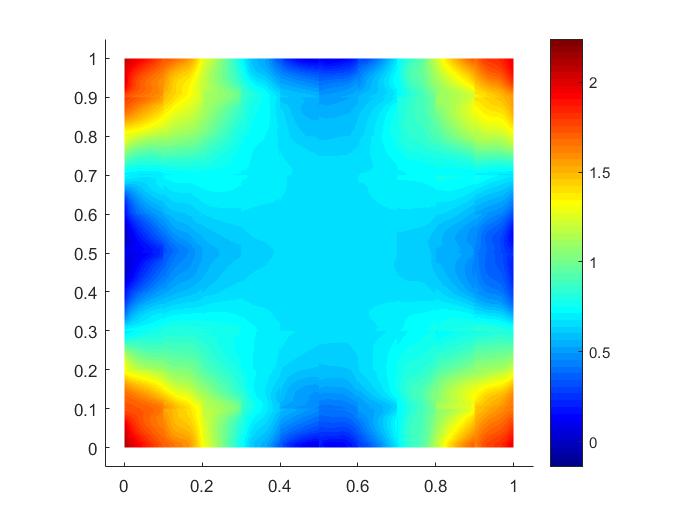}
	\end{minipage}%
	\begin{minipage}[t]{0.45\textwidth}
		\centering
		\includegraphics[width=3in]{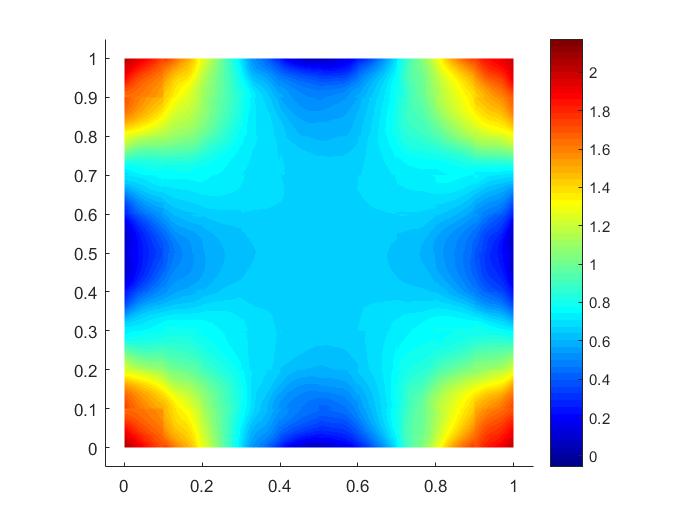}
	\end{minipage}%
	\caption{Fine solution and multisciale solution for Example 2. Top-Left: $u_{h,1}$. Top-Right: $\overline{u_h}$. Bottom-Left: $u_{H,1}$. Bottom-Right: $\overline{u_H}$. }
	\label{ex2 solution}
\end{figure}



\section*{Acknowledgements}
Q.~L.'s research is supported by NSF-DMS 1619778, 1750488, and UW-Madison Data Science Initiative.
E.~C.'s research is partially supported by Hong Kong RGC General Research Fund (Projects 14304217 and 14302018) and CUHK Direct Grant for Research 2018-19.

\bibliographystyle{plain}
\bibliography{references}

\end{document}